\documentclass{amsart}
\usepackage{amsmath}
\usepackage{amssymb}
\usepackage[all]{xy}
\usepackage{mathrsfs}
\usepackage{amsthm}
\usepackage{graphicx}
\usepackage{stmaryrd}

\theoremstyle{plain}
\newtheorem{thm}{Theorem}[section]
\newtheorem{lem}[thm]{Lemma}
\newtheorem{prop}[thm]{Proposition}
\newtheorem{cor}[thm]{Corollary}

\theoremstyle{definition}

\theoremstyle{remark}
\newtheorem{remark}[thm]{Remark}

\numberwithin{equation}{section}

\def\sheafHom{\mathcal{H} \hspace{-1pt} \mathit{om}}
\def\Av{{\rm A} \hspace{-1pt} {\rm v}}

\newcommand{\C}{\mathbb{C}}

\newcommand{\Z}{\mathbb{Z}}

\newcommand{\bbX}{\mathbb X}
\newcommand{\calA}{\mathcal{A}}
\newcommand{\calB}{\mathcal{B}}
\newcommand{\calC}{\mathcal{C}}
\newcommand{\calD}{\mathcal{D}}
\newcommand{\calE}{\mathcal{E}}
\newcommand{\calF}{\mathcal{F}}
\newcommand{\calG}{\mathcal{G}}
\newcommand{\calH}{\mathcal{H}}
\newcommand{\calI}{\mathcal{I}}
\newcommand{\calJ}{\mathcal{J}}
\newcommand{\calK}{\mathcal{K}}
\newcommand{\calL}{\mathcal{L}}
\newcommand{\calM}{\mathcal{M}}

\newcommand{\calO}{\mathcal{O}}
\newcommand{\calP}{\mathcal{P}}

\newcommand{\calR}{\mathcal{R}}
\newcommand{\calS}{\mathcal{S}}
\newcommand{\calT}{\mathcal{T}}

\newcommand{\calX}{\mathcal{X}}
\newcommand{\calY}{\mathcal{Y}}
\newcommand{\calZ}{\mathcal{Z}}
\newcommand{\wcalN}{\widetilde{\mathcal{N}}}
\newcommand{\wfrakg}{\widetilde{\mathfrak{g}}}

\newcommand{\frakn}{\mathfrak{n}}
\newcommand{\frakg}{\mathfrak{g}}
\newcommand{\frakt}{\mathfrak{t}}
\newcommand{\frakb}{\mathfrak{b}}

\newcommand{\frakK}{\mathfrak{K}}

\newcommand{\frakp}{\mathfrak{p}}

\newcommand{\lotimes}{{\stackrel{_L}{\otimes}}}
\newcommand{\rcap}{{\stackrel{_R}{\cap}}}
\newcommand{\Gm}{{\mathbb{G}}_{\mathbf{m}}}
\newcommand{\Hom}{{\rm Hom}}
\newcommand{\Ext}{{\rm Ext}}

\newcommand{\For}{{\rm For}}
\newcommand{\Id}{{\rm Id}}

\newcommand{\Coh}{\mathsf{Coh}}
\newcommand{\QCoh}{\mathsf{QCoh}}

\newcommand{\Mod}{\mathsf{-Mod}}

\newcommand{\qis}{{\rm qis}}

\newcommand{\fg}{{\rm fg}}

\newcommand{\op}{{\rm op}}

\newcommand{\scra}{\mathscr{A}}

\newcommand{\rmi}{\rm{(i)}}
\newcommand{\rmii}{\rm{(ii)}}
\newcommand{\rmiii}{\rm{(iii)}}
\newcommand{\rmiv}{\rm{(iv)}}
\newcommand{\rmv}{\rm{(v)}}
\newcommand{\rmvi}{\rm{(vi)}}
\newcommand{\aff}{{\rm aff}}

\newcommand{\Sym}{{\rm Sym}}

\newcommand{\Kth}{\mathsf{K}}

\newcommand{\bc}{\mathrm{bc}}

\begin{document}

\title{Iwahori--Matsumoto involution and linear Koszul Duality}
%
%
%
%

\begin{abstract}
We use linear Koszul duality, a geometric version of the standard duality between modules over symmetric and exte\-rior algebras studied in \cite{MR,MR2},
to give a geometric realization of the Iwahori--Matsumoto involution of affine Hecke algebras. More generally we prove that linear Koszul duality is compatible with convolution in a general context related to convolution algebras.
\end{abstract}

\author{Ivan Mirkovi\'c}
\address{University of Massachusetts, Amherst, MA.}
\email{mirkovic@math.umass.edu}

\author{Simon Riche}
\address{Clermont Universit{\'e}, Universit{\'e} Blaise Pascal, Laboratoire de Math{\'e}ma\-tiques, BP 10448, F-63000 Clermont-Ferrand. \newline
\indent CNRS, UMR 6620, Laboratoire de Math{\'e}matiques, F-63177 Aubi{\`e}re.}
\email{simon.riche@math.univ-bpclermont.fr}

\maketitle

\section{Introduction}

\subsection{Linear Koszul duality}

In \cite{MR, MR2} we have defined and studied \emph{linear Koszul duality}, a geometric version of the standard Koszul duality between (graded) modules over the symmetric algebra of a vector space $V$ and (graded) modules over the exterior algebra of the dual vector space $V^*$. As an application of this construction (in a particular case), given a vector bundle $A$ over a scheme $Y$ (satisfying a few technical conditions) and subbundles $A_1,A_2 \subset A$ we obtained an equivalence of triangulated categories between certain categories of coherent dg-sheaves on the derived intersections $A_1 \rcap_A A_2$ and $A_1^\bot \rcap_{A^*} A_2^\bot$. (Here $A^*$ is the dual vector bundle, and $A_1^\bot, A_2^\bot \subset A^*$ are the orthogonals to $A_1$ and $A_2$.)


\subsection{Linear Koszul duality of convolution algebras}
\label{ss:intro-notation}

In this paper we continue this study further in a special case related to convolution algebras (in the sense of \cite[\S 8]{CG}): we let $X$ be a smooth and proper complex algebraic variety, $V$ be a complex vector space, and $F \subset E:=V \times X$ be a subbundle. Applying our construction in the case $Y=X \times X$, $A=E \times E$, $A_1=\Delta V \times Y$ (where $\Delta V \subset V^2$ is the diagonal), $A_2=F \times F$ we obtain an equivalence between triangulated categories whose Grothendieck groups are respectively $\Kth^{\Gm}(F \times_V F)$ and $\Kth^{\Gm}(F^\bot \times_{V^*} F^\bot)$ (where $\Gm$ acts by dilatation along the fibers of $E \times E$ and $E^* \times E^*$). In fact we consider this situation more generally in the case $X$ is endowed with an action of a reductive group $G$, and $V$ is a $G$-module, and obtain in this way an isomorphism
\[
\Kth^{G \times \Gm}(F \times_V F) \ \cong \ \Kth^{G \times \Gm}(F^\bot \times_{V^*} F^\bot) \leqno(\star)
\]
where $G$ acts diagonally on $E \times E$ and $E^* \times E^*$.
(These constructions require an extension of the results of \cite{MR2} to the equivariant setting, treated in Section~\ref{sec:LKD}.)

Our main technical result is that this construction is compatible with convolution (even at the categorical level): the derived categories of dg-sheaves on our dg-schemes are endowed with a natural convolution product (which induces the usual convolution product of \cite{CG} at the level of $\Kth$-theory). We prove that our equivalence intertwines these products and sends the unit to the unit.

\subsection{A categorification of Iwahori--Matsumoto duality}
\label{ss:intro-IM}

We apply these results to give a geometric realization of the Iwahori--Matsumoto involution on the (extended) affine Hecke algebra $\calH_{\aff}$ of a semi-simple algebraic group $G$.

The Iwahori--Matsumoto involution of $\calH_{\aff}$ is a certain involution which naturally appears in the study of representations of the reductive $p$-adic group dual to $G$ in the sense of Langlands (see \emph{e.g.}~\cite{BaC, BaM}). This involution has a version for Lusztig's \emph{graded} affine Hecke algebra $\overline{\calH}_\aff$ associated with $\calH_\aff$ (\emph{i.e.}~the associated graded of $\calH_{\aff}$ for a certain filtration, see \cite{LAff}), which has been realized geometrically by S.~Evens and the first author in \cite{EM}. More precisely, $\overline{\calH}_\aff$ is isomorphic to the equivariant Borel--Moore homology of the Steinberg variety $Z$ of $G$ (\cite{LuCus1, LuCus2}), and it is proved in \cite{EM} that the Iwahori--Matsumoto involution is essentially given by a Fourier transform on this homology. 

In this paper we upgrade this geometric realization to the actual affine Hecke algebra $\calH_\aff$. This replaces Borel--Moore homology with $\Kth$-homology, and Fourier transform with linear Koszul duality. (Here we use Kazhdan--Lusztig geometric realization of $\calH_\aff$ via $\Kth$-homology \cite{KL}, see also \cite{CG}.) In the notation of \S\ref{ss:intro-notation} this geometric situation corresponds to the case $X=\calB$ (the flag variety of $G$), $V=\frakg^*$ (the co-adjoint representation), and $F=\wcalN$ (the Springer resolution): then $F \times_V F=Z$, and $F^\bot \times_{V^*} F^\bot$ is the ``extended Steinberg variety'', whose (equivariant) $\Kth$-homology is naturally isomorphic to that of $Z$, so that $(\star)$ indeed induces an automorphism of $\calH_\aff$.

In a sequel we will extend this result to a geometric realization of the Iwahori--Matsumoto involution of double affine Hecke algebras. One can expect that this would correspond to the Koszul duality
for the category $\mathcal{O}$ of the (still undefined) double loop group.


\subsection{The role of derived geometry}

Our ``categorification'' of the Iwahori--Matsumoto involution is a (contravariant) equivalence of categories $\frakK_{\mathrm{IM}}$ between some categories of equivariant coherent (dg-)sheaves on \emph{differential graded versions} of the varieties $Z$ and $\calZ$.
However, the codomain of $\frakK_{\mathrm{IM}}$ is ``not really derived,'' in the sense that it is equivalent to the usual category $\calD^b \Coh^{G \times \Gm}(\calZ)$, see Remark \ref{rk:calZ-derived}. On the other hand, the domain of $\frakK_{\mathrm{IM}}$ is \emph{not} equivalent to $\calD^b \Coh^{G \times \Gm}(Z)$; in fact the main reason for replacing $Z$ by its differential graded version is that the category $\calD^b \Coh^{G \times \Gm}(Z)$ has no obvious convolution product ``categorifying'' the product of $\calH_\aff$.
A similar replacement of $Z$ by its derived version appears in \cite{bezru}. See also Remark \ref{rk:conjecture-bezru} for a more elementary ``categorification'' statement, which however does not mention convolution.

In this paper, for simplicity and since these conditions are satisfied in our main example, we restrict ourselves to complex algebraic varieties endowed with an action of a reductive algebraic group. Using stacks it should be possible to work in a much more general setting; we do not consider this here.

\subsection{Linear Koszul duality and Fourier transform}

The proofs in this paper use compatibility properties of linear Koszul duality with various natural constructions proved in \cite{MR2}. (More precisely, here we need equivariant analogues of these results.) These properties are similar to well-known compatibility properties of the Fourier--Sato transform. We will make this observation precise in \cite{MR3}, showing that linear Koszul duality is related to the Fourier isomorphism in homology considered in \cite{EM} by the Chern character from $\Kth$-homology to (completed) Borel--Moore homology. This will explain the relation between Theorem~\ref{thm:mainthm} and the main result of \cite{EM}. The wish to upgrade Fourier transform to Koszul duality was the starting point of our work.

\subsection{Organization of the paper} 
In Section~\ref{sec:G-equiv} we collect some useful results on derived functors for equivariant dg-sheaves. 
In Section~\ref{sec:LKD} we extend the main results of \cite{MR2} to the equivariant setting. Most of the results in the rest of the paper will be formal consequences of these properties.
In Section~\ref{sec:convolution} we study the behavior of our linear Koszul duality equivalence in the context of convolution algebras. Finally, in Section~\ref{sec:IMinvolution} we prove that the special case of linear Koszul duality considered in \S\ref{ss:intro-IM} provides a geometric realization of the Iwahori--Matsumoto involution.


\subsection{Notation} 

Let $X$ be a complex algebraic variety\footnote{By \emph{complex algebraic variety} we mean a reduced, separated scheme of finite type over $\C$.} endowed with an action of an algebraic group $G$. We denote by $\QCoh^G(X)$, respectively $\Coh^G(X)$ the category of $G$-equivariant quasi-coherent, respectively coherent, sheaves on $X$. If $Y \subseteq X$ is a closed subscheme, we denote by $\Coh_Y^G(X)$ the full subcategory of $\Coh^G(X)$ whose objects are supported set-theoretically on $Y$.  If $\calF$, $\calG$ are sheaves of $\calO_X$-modules, we denote by $\calF \boxplus \calG$ the $\calO_{X^2}$-module $(p_1)^* \calF \oplus (p_2)^* \calG$ on $X^2$, where $p_1, p_2 : X \times X \to X$ are the first and second projections.

We will frequently work with $\Z^2$-graded sheaves $\calM$. The $(i,j)$ component of $\calM$ will be denoted $\calM^i_j$. Here ``$i$'' will be called the cohomological grading, and ``$j$'' will be called the internal grading. We will write $|m|$ for the cohomological degree of a homogeneous local section $m$ of $\calM$. Ordinary sheaves will be considered as $\Z^2$-graded sheaves concentrated in bidegree $(0,0)$. As usual, if $\calM$ is a $\Z^2$-graded sheaf of $\calO_X$-modules, we denote by $\calM^\vee$ the $\Z^2$-graded $\calO_X$-module such that
\[
(\calM^\vee)^i_j \ := \ \sheafHom_{\calO_X}(\calM^{-i}_{-j},\calO_X).
\]
We will denote by $\langle m \rangle$ the shift in the internal grading, such that
\[
(\calM \langle m \rangle)^i_j \ := \ \calM^i_{j-m}.
\]

We will consider $G \times \Gm$-equivariant\footnote{Note that $G$ and $\Gm$ do not play the same role: $G$ acts on the variety, while the $\Gm$-equivariance simply means an extra $\Z$-grading.} sheaves of quasi-coherent $\calO_X$-dg-algebras (as in \cite{MR}). Recall that such an object is a $\Z^2$-graded sheaf of $\calO_X$-dg-algebras, endowed with a differential of bidegree $(1,0)$ of square $0$ which satisfies the Leibniz rule with respect to the cohomological grading, and also endowed with the structure of a $G$-equivariant quasi-coherent sheaf, compatible with all other structures. If $\calA$ is such a dg-algebra, we denote by $\calC(\calA\Mod^G)$ the category of $G \times \Gm$-equivariant quasi-coherent sheaves of $\calO_X$-dg-modules over $\calA$. (Such an object is a $\Z^2$-graded $G$-equivariant quasi-coherent $\calO_X$-module, endowed with a differential of bidegree $(1,0)$ and an action of $\calA$ -- extending the $\calO_X$-action -- which makes it an $\calA$-dg-module, the action map being graded and $G$-equivariant.) We denote by $\calD(\calA\Mod^G)$ the associated derived category, obtained by inverting quasi-isomorphisms.

If $\calF$ is an $\calO_X$-modules (considered as a bimodule where the left and right actions coincide), we denote by $\mathrm{S}_{\calO_X}(\calF)$, respectively $\bigwedge_{\calO_X}(\calF)$, the symmetric, respectively exterior, algebra of $\calF$, \emph{i.e.}~the quotient of the tensor algebra of $\calF$ by the relations $f \otimes g - g \otimes f$, respectively $f \otimes g + g \otimes f$, for $f,g$ local sections of $\calF$. If $\calF$ is a complex of (graded) $\calO_X$-modules, we denote by $\Sym_{\calO_X}(\calF)$ the graded-symmetric algebra of $\calF$, \emph{i.e.}~the quotient of the tensor algebra of $\calF$ by the relations $f \otimes g - (-1)^{|f| \cdot |g|}g \otimes f$ for $f,g$ homogeneous local sections of $\calF$. This algebra is a sheaf of ($\Gm$-equivariant) $\calO_X$-dg-algebras in a natural way.

As in \cite{MR2} we use the general convention that we denote by the same symbol a functor and the induced functor between opposite categories.


\subsection{Acknowledgements} 
This article is a sequel to \cite{MR, MR2}. It was started while both authors were members of the Institute for Advanced Study in Princeton. Part of this work was done while the second author was a student at Paris 6 University, and while he visited the Massachusetts Institute of Technology.

We thank Roman Bezrukavnikov for helpful remarks, and the referee for his careful reading.


I.M.~was supported by NSF grants.
S.R.~was supported by ANR grants No.~ANR-09-JCJC-0102-01 and No.~ANR-10-BLAN-0110.


\section{Functors for $G$-equivariant Quasi-coherent Sheaves}
\label{sec:G-equiv}

This section collects general results on equivariant quasi-coherent sheaves and dg-modules. Some of these results are well known, but we include proofs since we could not find an appropriate reference. Most of our assumptions are probably not necessary, but they are satisfied in the situations where we want to apply these results.

\subsection{Equivariant Grothendieck--Serre duality}
\label{ss:duality-equiv}

Let $X$ be a complex algebraic variety, endowed with an action of a reductive algebraic group $G$. By \cite[Example 2.16]{AB}, there exists an object $\Omega$ in $\calD^b \Coh^G(X)$ whose image under the forgetful functor $\For : \calD^b \Coh^G(X) \to \calD^b \Coh(X)$ is a dualizing complex. We will fix such an object.

We will sometimes make the following additional assumption:
\begin{equation}
\label{eqn:enough-flats}
\begin{array}{c}
\text{For any } \calF \text{ in } \Coh^G(X), \text{ there exists } \calP \text{ in } \Coh^G(X) \\
\text{which is flat over } \calO_X \text{ and a surjection } \calP \twoheadrightarrow \calF \text{ in } \Coh^G(X).
\end{array}
\end{equation}
This assumption is standard in this setting; it is satisfied \emph{e.g.}~if $X$ is normal and quasi-projective (see \cite[Proposition 5.1.26]{CG}), or if $X$ admits an ample family of line bundles in the sense of \cite[Definition 1.5.3]{VV}. Note also that~\eqref{eqn:enough-flats} implies a similar property for \emph{quasi}-coherent sheaves.

Recall that by \cite[Corollary 2.11]{AB} the natural functors
\[
\calD^b \Coh^G(X) \to \calD^b \QCoh^G(X) \quad \text{and} \quad \calD^b \Coh(X) \to \calD^b \QCoh(X)
\]
are both fully faithful. This will allow us not to distinguish between morphisms in these categories.

We denote by $a: G \times X \to X$ the action, and by $p : G \times X \to X$ the projection. Both of these morphisms are flat and affine. Recall the ``averaging functor"
\[
\Av : \left\{
\begin{array}{ccc}
\QCoh(X) & \to & \QCoh^G(X) \\
\calF & \mapsto & a_* p^* \calF
\end{array}
\right. .
\]
This functor is exact, and is right adjoint to the forgetful functor $\For : \QCoh^G(X) \to \QCoh(X)$ which is exact. Hence $\Av$ sends injective objects of $\QCoh(X)$ to injective objects of $\QCoh^G(X)$. From this one easily deduces that there are enough injective objects in $\QCoh^G(X)$, and that every such injective object is a direct summand of an injective object of the form $\Av(\calI)$ for some injective $\calI$ in $\QCoh(X)$.

Recall also that for any $\calF$ in $\Coh^G(X)$ and $\calG$ in $\QCoh^G(X)$, the $\C$-vector space $\Hom_{\calO_X}(\calF,\calG)$ is naturally an algebraic $G$-module, and that we have a canonical isomorphism 
\begin{equation}
\label{eqn:morphisms-equiv}
\Hom_{\QCoh^G(X)}(\calF,\calG) \ \cong \ \bigl( \Hom_{\calO_X}(\calF,\calG) \bigr)^G
\end{equation}
induced by the functor $\For$.
(Here and below, for simplicity we do not write the functor $\For$.)
Now we prove a version of this statement for derived categories, which will simplify our constructions a lot.

\begin{lem}
\label{lem:morphisms-equiv}
For any $\calF,\calG$ in $\calD^b \Coh^G(X)$, the $\C$-vector space $\Hom_{\calD^b \Coh(X)}(\calF,\calG)$ is naturally an algebraic $G$-module. Moreover, $\For$ induces an isomorphism $\Hom_{\calD^b \Coh^G(X)}(\calF,\calG) \ \cong \ \bigl( \Hom_{\calD^b \Coh(X)}(\calF,\calG) \bigr)^G$.
\end{lem}

\begin{proof}
The construction of the $G$-action is standard, and left to the reader. To prove the isomorphism, by a standard ``d{\'e}vissage'' argument it is enough to prove that if $\calF$ and $\calG$ are in $\Coh^G(X)$ and if $i \geq 0$ the natural morphism
\[
\Ext^i_{\Coh^G(X)}(\calF,\calG) \ \to \ \bigl( \Ext^i_{\Coh(X)}(\calF,\calG) \bigr)^G
\]
is an isomorphism. Now let $\calI$ be an injective resolution of $\calG$ in the abelian category $\QCoh^G(X)$. By Lemma~\ref{lem:ext-vanishing-equiv} below, this complex is acyclic for the functor $\Hom_{\calO_X}(\calF,-)$, hence can be used to compute $\Ext^i_{\Coh(X)}(\calF,\calG)$. Then our claim easily follows from isomorphism~\eqref{eqn:morphisms-equiv} and the fact that the functor of $G$-invariants is exact.
\end{proof}

\begin{lem}
\label{lem:ext-vanishing-equiv}
Let $\calF$ be in $\Coh^G(X)$, and $\calJ$ be an injective object of  $\QCoh^G(X)$. Then for any $j>0$ we have $\Ext^j_{\QCoh(X)}(\calF,\calJ)=0$.
\end{lem}

\begin{proof}
We can assume that $\calJ=\Av(\calI)$ for some injective object $\calI$ of $\QCoh(X)$. Then we have
\[
\Ext^j_{\QCoh(X)}(\calF,\calJ) = \Ext^j_{\QCoh(X)}(\calF,a_* p^*\calI) \cong \Ext^j_{\QCoh(G \times X)}(a^*\calF,p^*\calI)
\]
by adjunction. As $\calF$ is $G$-equivariant we have an isomorphism $a^* \calF \cong p^* \calF$, and using adjunction again we deduce an isomorphism
\[
\Ext^j_{\QCoh(X)}(\calF,\calJ) \cong \Ext^j_{\QCoh(X)}(\calF,p_* p^*\calJ).
\]
Now we have $p_* p^* \calJ \cong \C[G] \otimes_\C \calJ$, and it follows from \cite[Corollary II.7.9]{H} that $p_* p^* \calJ$ is injective. This finishes the proof.
\end{proof}

As $\Omega$ is a dualizing complex, we have an equivalence
\[
\mathrm{D}_\Omega \ := \ R\sheafHom_{\calO_X}(-,\Omega) : \calD^b \Coh(X) \xrightarrow{\sim} \calD^b \Coh(X)^\op,
\]
and a canonical isomorphism of functors $\varepsilon_\Omega : \Id_{\calD^b \Coh(X)} \to \mathrm{D}_\Omega \circ \mathrm{D}_\Omega$
(see \emph{e.g.}~\cite[\S1.5]{MR2} for details).

Let now $\calI_\Omega$ be a bounded below complex of injective objects of $\QCoh^G(X)$ whose image in the derived category $\calD^+ \QCoh^G(X)$ is $\Omega$. Then the ``internal Hom'' bifunctor defines a functor
\[
{}^0 \mathrm{D}_\Omega^{G} := \sheafHom_{\calO_X}(-,\calI_\Omega) : \calC^b \Coh^G(X) \to \calC^+ \QCoh^G(X)^\op.
\]

\begin{lem}
\label{lem:duality-equiv}
Assume condition~\eqref{eqn:enough-flats} is satisfied.

The functor ${}^0 \mathrm{D}_\Omega^{G}$ is exact. The induced functor on derived categories factors through a functor between bounded derived categories
\[
\mathrm{D}_\Omega^{G} : \calD^b \Coh^G(X) \to \calD^b \Coh^G(X)^\op.
\]
Moreover, the following diagram commutes up to isomorphism:
\[
\xymatrix@C=2cm@R=0.6cm{
\calD^b \Coh^G(X) \ar[r]^-{\mathrm{D}_\Omega^{G}} \ar[d]_-{\For} & \calD^b \Coh^G(X)^\op \ar[d]^-{\For} \\
\calD^b \Coh(X) \ar[r]^-{\mathrm{D}_\Omega} & \calD^b \Coh(X)^\op
}
\]
where vertical arrows are the usual forgetful functors.
\end{lem}

\begin{proof}
To prove exactness, it suffices to prove that if $\calJ$ is an injective object of $\QCoh^G(X)$, the functor
\[
\sheafHom_{\calO_X}(-,\calJ) : \Coh^G(X) \to \QCoh^G(X)^\op
\]
is exact. One can assume that $\calJ=\Av(\calI)$ for some injective $\calI$ in $\QCoh(X)$. Then for any $\calF$ in $\Coh^G(X)$ we have
\[
\sheafHom_{\calO_X}(\calF,\calJ) = \sheafHom_{\calO_X}(\calF,a_* p^* \calI) \cong a_* \sheafHom_{\calO_{G \times X}}(a^*\calF, p^* \calI)
\]
by adjunction. Now we have a canonical isomorphism $a^* \calF \cong p^*\calF$, and we deduce isomorphisms
\[
\sheafHom_{\calO_X}(\calF,\calJ) \cong a_* \sheafHom_{\calO_{G \times X}}(p^*\calF, p^* \calI) \cong a_* p^* \sheafHom_{\calO_{X}}(\calF, \calI) \cong \Av \bigl( \sheafHom_{\calO_{X}}(\calF, \calI) \bigr).
\]
As both the functors $\sheafHom_{\calO_X}(-,\calI)$ and $\Av$ are exact, we deduce the claim, hence exactness of ${}^0 \mathrm{D}_\Omega^{G}$.

Let us denote by ${}' \mathrm{D}_\Omega^{G}$ the functor induced between derived categories, and by ${}' \mathrm{D}_\Omega$ the non-equivariant analogue. Now, let us prove that the following diagram commutes:
\begin{equation}
\label{eqn:diagram-duality-equiv}
\vcenter{
\xymatrix@C=2cm@R=0.6cm{
\calD^b \Coh^G(X) \ar[r]^-{{}' \mathrm{D}_\Omega^{G}} \ar[d]_-{\For} & \calD^+ \QCoh^G(X)^\op \ar[d]^-{\For} \\
\calD^b \Coh(X) \ar[r]^-{{}' \mathrm{D}_\Omega} & \calD^+ \QCoh(X)^\op
}
}
\end{equation}
Let $\calJ_\Omega$ be a complex of injective objects in $\QCoh(X)$ whose image in $\calD^+ \QCoh(X)$ is $\Omega$, so that the functor ${}' \mathrm{D}_\Omega$ is the functor induced by the exact functor $\sheafHom_{\calO_X}(-,\calJ_\Omega) : \calC^b \Coh(X) \to \calC^+ \QCoh(X)^\op$. By standard arguments there exists a quasi-isomorphism $\calI_\Omega \xrightarrow{\qis} \calJ_\Omega$ in the category $\calC^+ \QCoh(X)$. We denote by $\calK_\Omega$ the cone of this morphism. To prove the commutativity it is sufficient to prove that for any $\calF$ in $\calC^b \Coh^G(X)$ the natural morphism
\[
\sheafHom_{\calO_X}(\calF,\calI_\Omega) \to \sheafHom_{\calO_X}(\calF,\calJ_\Omega)
\]
is a quasi-isomorphism, or in other words that the complex $\sheafHom_{\calO_X}(\calF,\calK_\Omega)$ is acyclic. By our assumption~\eqref{eqn:enough-flats}, there exists a complex $\calL$ in $\calC^- \Coh^G(X)$ whose objects are $\calO_X$-flat and a quasi-isomorphism $\calL \xrightarrow{\qis} \calF$. Using what was checked in the first paragraph of this proof, one can show that the induced morphism $\sheafHom_{\calO_X}(\calF,\calK_\Omega) \to\sheafHom_{\calO_X}(\calL,\calK_\Omega)$ is a quasi-isomorphism. Now, as $\calK_\Omega$ is an acyclic complex and $\calL$ a bounded above complex of flat $\calO_X$-modules the complex $\sheafHom_{\calO_X}(\calL,\calK_\Omega)$ is acyclic. This finishes the proof of the commutativity of~\eqref{eqn:diagram-duality-equiv}.

Finally, as the functor ${}' \mathrm{D}_\Omega$ takes values in $\calD^b \Coh(X)$, one deduces the second claim of the lemma and the commutativity of the diagram from the commutativity of~\eqref{eqn:diagram-duality-equiv}.
\end{proof}

\begin{cor}
\label{cor:duality-equiv}
Assume condition~\eqref{eqn:enough-flats} is satisfied.

There exists a canonical isomorphism $\Id \xrightarrow{\sim} \mathrm{D}_\Omega^{G} \circ \mathrm{D}_\Omega^{G}$ of endofunctors of $\calD^b \Coh^G(X)$. In particular,  
$\mathrm{D}_\Omega^{G}$ is an equivalence of categories.
\end{cor}

\begin{proof}
For $\calF$ in $\calD^b \Coh^G(X)$, the isomorphism $\varepsilon_\Omega(\calF) : \calF \xrightarrow{\sim} \mathrm{D}_\Omega \circ \mathrm{D}_\Omega(\calF)$ is canonical, and in particular invariant under the action of $G$ on $\Hom_{\calD^b \Coh(X)}(\calF, \mathrm{D}_\Omega \circ \mathrm{D}_\Omega(\calF))$. Using the commutativity of the diagram in Lemma~\ref{lem:duality-equiv} and Lemma~\ref{lem:morphisms-equiv}, we deduce the existence of the canonical isomorphism $\Id \xrightarrow{\sim} \mathrm{D}_\Omega^{G} \circ \mathrm{D}_\Omega^{G}$. The final claim is obvious.
\end{proof}

\subsection{Grothendieck--Serre duality in the dg setting}
\label{ss:duality-dg-equiv}

As above let $X$ be a complex algebraic variety, endowed with an action of a reductive algebraic group $G$. Let also $\calA$ be a $G \times \Gm$-equivariant,\footnote{The $\Gm$-equivariance (or equivalently the additional $\Z$-grading) will not play any role in \S\S\ref{ss:duality-dg-equiv}--\ref{ss:direct-image-equiv}. We make this assumption to avoid introducing new notation, and because this is the only case we will considerx.} non-positively (cohomologically) graded, graded-commutative, sheaf of quasi-coherent $\calO_X$-dg-algebras. We assume furthermore that $\calA$ is locally finitely generated over $\calA^0$, that $\calA^0$ is locally finitely generated as an $\calO_X$-algebra, and that $\calA$ is $K$-flat as a $\Gm$-equivariant $\calA^0$-dg-module (in the sense of \cite{Sp}). Finally, we will assume that condition~\eqref{eqn:enough-flats} is satisfied.

If $A$ denotes the ($G$-equivariant) affine scheme over $X$ such that the pushforward of $\calO_A$ to $X$ is $\calA^0$, then there exists a $\Gm$-equivariant quasi-coherent $G$-equivariant $\calO_A$-dg-algebra $\calA'$ whose direct image to $X$ is $\calA$. Moreover there exists an equivalence of categories $\calC(\calA'\Mod^G) \cong \calC(\calA\Mod^G)$. Using this trick we can reduce our situation to the case $\calA$ is $\calO_X$-coherent and $K$-flat as an $\calO_X$-dg-module.

Using conventions similar to those in \cite{MR2},
we denote by $\calD^\bc(\calA\Mod^G)$ the subcategory of $\calD(\calA\Mod^G)$ whose objects are the dg-modules $\calM$ such that, for any $j \in \Z$, the complex $\calM_j$ has bounded and coherent cohomology.

\begin{lem}
\label{lem:morphisms-dg-equiv}
For any $\calF,\calG$ in $\calD^\bc(\calA\Mod^G)$ the $\C$-vector space $\Hom_{\calD^\bc(\calA\Mod)}(\calF,\calG)$ has a natural structure of an algebraic $G$-module. Moreover, the natural morphism
\begin{equation}
\label{eqn:morphisms-dg-equiv}
\Hom_{\calD^\bc(\calA\Mod^G)}(\calF,\calG) \ \to \ \bigl( \Hom_{\calD^\bc(\calA\Mod)}(\calF,\calG) \bigr)^G
\end{equation}
induced by the forgetful functor is an isomorphism.
\end{lem}

\begin{proof}
The construction of the $G$-action is similar the the one in \S\ref{ss:duality-equiv}. Now, let us prove that~\eqref{eqn:morphisms-dg-equiv} is an isomorphism. As explained above, we can assume that $\calA$ is $\calO_X$-coherent and $K$-flat as an $\calO_X$-dg-module. Consider the induction functor
\[
\mathrm{Ind}_\calA : \left\{
\begin{array}{ccc}
\calC(\calO_X\Mod^G) & \to & \calC(\calA\Mod^G) \\
\calM & \mapsto & \calA \otimes_{\calO_X} \calM
\end{array}
\right. .
\]
This functor is left adjoint to the forgetful functor $\For_\calA : \calC(\calA\Mod^G) \to \calC(\calO_X\Mod^G)$. Moreover, by our $K$-flatness assumption the functor $\mathrm{Ind}_\calA$ is exact, hence induces a functor between derived categories, which we denote similarly. Then the functors
\[
\mathrm{Ind}_\calA : \calD(\calO_X\Mod^G) \to \calD(\calA\Mod^G) \quad \text{and} \quad \For_\calA : \calD(\calA\Mod^G) \to \calD(\calO_X\Mod^G)
\]
are again adjoint. As $\calA$ is $\calO_X$-coherent, the functor $\mathrm{Ind}_\calA$ sends the subcategory $\calD^\bc(\calO_X\Mod^G)$ into $\calD^\bc(\calA\Mod^G)$.

Using these remarks and Lemma~\ref{lem:morphisms-equiv}, one can show that~\eqref{eqn:morphisms-dg-equiv} is an isomorphism in the case $\calF=\mathrm{Ind}_\calA(\calF')$ for some object $\calF'$ in $\calD^\bc(\calO_X\Mod^G)$. Now we explain how to reduce the general case to this case. In fact, using a simplified form of the construction in \cite[proof of Theorem 1.3.3]{R2} (without taking $K$-flat resolutions), one can check that any object of $\calD^\bc(\calA\Mod^G)$ is a direct limit of a family $(\calP_p)_{p \geq 0}$ of objects of $\calD^\bc(\calA\Mod^G)$ such that each $\calP_p$ admits a finite filtration with subquotients of the form $\mathrm{Ind}_\calA(\calH)$ for some $\calH$ in $\calD^\bc(\calO_X\Mod^G)$. As the functor of $G$-fixed points commutes with inverse limits, this reduces the general case to the case treated above, and finishes the proof.
\end{proof}

Finally we can prove our ``duality'' statement for $G$-equivariant $\calA$-dg-modules. First, recall that there is a canonical equivalence of triangulated categories
\[
\mathrm{D}_\Omega^\calA : \calD^\bc(\calA\Mod) \xrightarrow{\sim} \calD^\bc(\calA\Mod)^\op,
\] 
where the exponent ``$\bc$'' has the same meaning as above, see \cite[\S 1.5]{MR2} for details.
The equivariant analogue of this statement can be deduced from the properties of the functor $\mathrm{D}_\Omega^\calA$ using Lemma~\ref{lem:morphisms-dg-equiv}, just as the properties of the functor $\mathrm{D}_\Omega^{G}$ where deduced from those of $\mathrm{D}_\Omega$ in \S\ref{ss:duality-equiv}.

\begin{prop}
\label{prop:duality-dg-equiv}
There exists an equivalence of categories
\[
\mathrm{D}_\Omega^{\calA,G} : \calD^\bc(\calA\Mod^G) \xrightarrow{\sim} \calD^\bc(\calA\Mod^G)^\op
\]
such that the following diagram commutes (where the vertical arrows are the natural forgetful functors):
\[
\xymatrix@C=2cm@R=0.6cm{
\calD^\bc(\calA\Mod^G) \ar[r]^-{\mathrm{D}_\Omega^{\calA,G}}_-{\sim} \ar[d]_-{\For} & \calD^\bc(\calA\Mod^G)^\op \ar[d]^-{\For} \\
\calD^\bc(\calA\Mod) \ar[r]^-{\mathrm{D}_\Omega^{\calA}}_-{\sim} & \calD^\bc(\calA\Mod)^\op
}
\]
and a canonical isomorphism of functors $\Id \xrightarrow{\sim} \mathrm{D}_\Omega^{\calA,G} \circ \mathrm{D}_\Omega^{\calA,G}$.
\end{prop}

\subsection{Inverse image of equivariant dg-sheaves}
\label{ss:inverse-image-equiv}

We let $X$ and $Y$ be complex algebraic varieties, each endowed with an action of an algebraic group $G$. (In practice $G$ will be reductive, as above, but this property will not be used in this subsection.) We also assume that condition~\eqref{eqn:enough-flats} is satisfied on $Y$.

Let $\calA$, respectively $\calB$, be a sheaf of non-positively graded, graded-commutative, quasi-coherent, $G \times \Gm$-equivariant $\calO_X$-dg-algebras, respectively $\calO_Y$-dg-algebras. Let $f : (X,\calA) \to (Y,\calB)$ be a $G \times \Gm$-equivariant morphism of dg-ringed spaces.
By \cite[Proposition 3.2.2]{BR} (see also \cite[\S 1.1]{MR2}) the inverse image functor
\[
f^* : \calC(\calB \Mod) \to \calC(\calA \Mod)
\]
admits a left derived functor
\begin{equation}
\label{eqn:Lf}
Lf^* : \calD(\calB \Mod) \to \calD(\calA \Mod).
\end{equation}
This property follows from the existence of $K$-flat resolutions in $\calC(\calB\Mod)$. The same arguments extend to the $G$-equivariant setting under our assumption that condition~\eqref{eqn:enough-flats} holds on $Y$. 

\begin{lem}
\label{lem:Kflats}
For any object $\calM$ in $\calC(\calB \Mod^G)$, there exists an object $\calP$ in $\calC(\calB \Mod^G)$, which is $K$-flat as a $\calB$-dg-module, and a quasi-isomorphism of $G \times \Gm$-equivariant $\calB$-dg-modules $\calP \xrightarrow{\qis} \calF$.
\end{lem}

In particular, it follows from this lemma that the functor
\[ 
f^*: \calC(\calB \Mod^G) \to \calC(\calA \Mod^G)
\]
admits a left derived functor
\[
Lf^* : \calD(\calB\Mod^G) \to \calD(\calA\Mod^G).
\]
Moreover, the following diagram commutes by definition: 
\begin{equation}
\label{eqn:diagram-inverse-image}
\vcenter{
\xymatrix@C=2cm@R=0.6cm{
\calD(\calB\Mod^G) \ar[d]_-{\For} \ar[r]^-{Lf^*} & \calD(\calA\Mod^G) \ar[d]^-{\For} \\
\calD(\calB\Mod) \ar[r]^-{\eqref{eqn:Lf}} & \calD(\calA\Mod). 
}
}
\end{equation}
(This property justifies our convention that the notation $Lf^*$ denotes the derived functor both in the equivariant and non-equivariant settings.)

\subsection{Direct image of equivariant dg-sheaves}
\label{ss:direct-image-equiv}

We let again $X$ and $Y$ be complex algebraic varieties, each endowed with an action of an algebraic group $G$. Let $\calA$, respectively $\calB$, be a sheaf of non-positively graded, graded-commutative, $G \times \Gm$-equivariant, quasi-coherent $\calO_X$-dg-algebras, respectively $\calO_Y$-dg-algebras. Let $f : (X,\calA) \to (Y,\calB)$ be a $G \times \Gm$-equivariant morphism of dg-ringed spaces.

We will assume that $\calA$ is locally free of finite rank over $\calA^0$, that $\calA^0$ is locally finitely generated as an $\calO_X$-algebra, and finally that $\calA$ is $K$-flat as a $\Gm$-equivariant $\calA^0$-dg-module.

It follows from \cite[Proposition 3.3.2]{BR} (existence of $K$-injective resolutions in $\calC(\calA\Mod)$), see also \cite[\S 1.1]{MR2}, that the direct image functor
\[
f_* : \calC(\calA\Mod) \to \calC(\calB\Mod)
\]
admits a right derived functor
\begin{equation}
\label{eqn:direct-image-non-equiv}
Rf_* : \calD(\calA\Mod) \to \calD(\calB\Mod).
\end{equation}
Our goal in this subsection is to extend this property to the equivariant setting. In this case for simplicity we restrict to a subcategory.

We denote by $\calC^+(\calA\Mod^G)$ the subcategory of $\calC(\calA\Mod^G)$ whose objects are bounded below in the cohomological grading (uniformly in the internal grading), and by $\calD^+(\calA\Mod^G)$ the full subcategory of $\calD(\calA\Mod^G)$ whose objects are the dg-modules whose cohomology is bounded below. Note that the natural functor from the derived category associated with $\calC^+(\calA\Mod^G)$ to $\calD(\calA\Mod^G)$ is fully faithful, with essential image $\calD^+(\calA\Mod^G)$.

\begin{prop}
\label{prop:direct-image-equiv}
\begin{enumerate}
\item
For any $\calM$ in  $\calC^+(\calA\Mod^G)$, there exists an object $\calI$ in $\calC^+(\calA\Mod^G)$ which is $K$-injective in $\calC(\calA\Mod^G)$ and a quasi-isomorphism $\calM \xrightarrow{\qis} \calI$.
\item
The functor $f_* : \calC^+(\calA\Mod^G) \to \calC(\calB\Mod^G)$ admits a right derived functor
\[
Rf_* : \calD^+(\calA\Mod^G) \to \calD(\calB\Mod^G).
\]
Moreover, the following diagram is commutative up to isomorphism:
\[
\xymatrix@C=2cm@R=0.6cm{
\calD^+(\calA\Mod^G) \ar[r]^-{Rf_*} \ar[d]_-{\For} & \calD(\calB\Mod^G) \ar[d]^-{\For} \\
\calD(\calA\Mod) \ar[r]^-{\eqref{eqn:direct-image-non-equiv}} & \calD(\calB\Mod).
}
\]
\end{enumerate}
\end{prop}

\begin{proof}
$(1)$ As in \S\ref{ss:duality-dg-equiv} we can assume that $\calA^0=\calO_X$, so that $\calA$ is locally free of finite rank over $\calO_X$. Under this assumption, the proof of \cite[Lemma 1.3.5]{R2} (using the coinduction functor $\mathrm{Coind}_\calA$) generalizes directly to our setting and proves property $(1)$.

$(2)$ The existence of the derived functor follows from $(1)$. Now, let us prove the commutativity of our diagram. Again we can assume that $\calA^0=\calO_X$. As explained in \cite[Proof of Lemma 1.5.9]{VV}, any injective object of $\QCoh^G(X)$ is $f_*$-acyclic. It follows easily from this that the diagram commutes if $\calA=\calO_X$ and $\calB=\calO_Y$. Using this and \cite[Proposition 3.3.6]{BR} (compatibility of derived direct images for $\calA$- and $\calO_X$-dg-modules), it suffices to prove that the following diagram commutes:
\[
\xymatrix@C=2cm@R=0.6cm{
\calD^+(\calA\Mod^G) \ar[r]^-{Rf_*} \ar[d]_-{\For_\calA} & \calD(\calB\Mod^G) \ar[d]^-{\For_\calB} \\
\calD^+(\calO_X\Mod^G) \ar[r]^-{R(f_0)_*} & \calD(\calO_Y\Mod^G)
}
\]
(where $f_0 : X \to Y$ is the morphism of schemes underlying $f$.) However, this is clear from the construction in $(1)$ and the fact that, under our assumptions, the functor $\mathrm{Coind}_\calA$ sends a bounded below complex of injective objects of $\QCoh^G(X)$ to a complex which has the same property.
\end{proof}

\section{Linear Koszul Duality}
\label{sec:LKD}

\subsection{Linear Koszul duality in the equivariant setting}
\label{ss:lkd-equivariant}

We let $X$ be a complex algebraic variety, endowed with an action of a reductive algebraic group $G$, and let $\Omega$ be an object in $\calD^b \Coh^G(X)$ whose image in $\calD^b \Coh(X)$ is a dualizing complex (see \S\ref{ss:duality-equiv}). We will assume that condition~\eqref{eqn:enough-flats} is satisfied.

Let $E$ be a $G$-equivariant vector bundle on $X$ and let $F_1,F_2 \subset E$ be $G$-equivariant subbundles. As in \cite{MR2}, we denote by $\calE,\calF_1,\calF_2$ the sheaves of sections of $E, F_1, F_2$, and we define the $G \times \Gm$-equivariant complexes
\[
\calX :=(0 \to \calF_1^\bot \to \calF_2^\vee \to 0), \quad \calY:=(0 \to \calF_2 \to \calE/\calF_1 \to 0).
\]
In $\calX$, $\calF_1^\bot$ is in bidegree $(-1,2)$, $\calF_2^\vee$ is in bidegree $(0,2)$, and the differential is the composition of the natural maps $\calF_1^\bot \hookrightarrow \calE^\vee \twoheadrightarrow \calF_2^\vee$. In $\calY$, $\calF_2$ is in bidegree $(1,-2)$, $\calE/\calF_1$ is in bidegree $(2,-2)$, and the differential is the opposite of the composition of the natural maps $\calF_2 \hookrightarrow \calE \twoheadrightarrow \calE/\calF_1$. We will work with the $G \times \Gm$-equivariant sheaves of dg-algebras
\[
\calT:=\Sym(\calX), \quad \calS:=\Sym(\calY), \quad \calR:=\Sym(\calY[2]).
\]

We set
\[
\calD^{\mathrm{c}}_{G \times \Gm}(F_1 \, \rcap_E \, F_2) \ := \ \calD^{\fg}(\calT\Mod^G), \quad
\calD^{\mathrm{c}}_{G \times \Gm}(F_1^{\bot} \, \rcap_{E^*} \, F_2^{\bot}) \ := \ \calD^{\fg}(\calR\Mod^G),
\]
where the exponent ``$\fg$'' means the subcategory of dg-modules over $\calT$ (or $\calR$) whose cohomology is locally finitely generated over $\calH^\bullet(\calT)$ (or $\calH^\bullet(\calR)$). We also set
\[
\calD^{\mathrm{c}}_{\Gm}(F_1 \, \rcap_E \, F_2) \ := \ \calD^{\fg}(\calT\Mod), \quad
\calD^{\mathrm{c}}_{\Gm}(F_1^{\bot} \, \rcap_{E^*} \, F_2^{\bot}) \ := \ \calD^{\fg}(\calR\Mod).
\]
(Here and below, when no group appears in the notation, we mean the categories of $\{1\} \times \Gm$-equivariant dg-modules, \emph{i.e.}~we forget the action of $G$.) Recall that in \cite[Theorem 1.9.3]{MR2} we have constructed an equivalence of triangulated categories
\[
\kappa_{\Omega} : \calD^{\mathrm{c}}_{\Gm}(F_1 \, \rcap_E \, F_2) \ \xrightarrow{\sim} \ \calD^{\mathrm{c}}_{\Gm}(F_1^{\bot} \, \rcap_{E^*} \, F_2^{\bot})^\op.
\]
The following result is an equivariant analogue of this equivalence.

\begin{thm}
\label{thm:lkd-equivariant}
There exists an equivalence of triangulated categories
\[
\kappa_{\Omega}^G : \calD^{\mathrm{c}}_{G \times \Gm}(F_1 \, \rcap_E \, F_2) \ \xrightarrow{\sim} \ \calD^{\mathrm{c}}_{G \times \Gm}(F_1^{\bot} \, \rcap_{E^*} \, F_2^{\bot})^\op,
\]
which satisfies $\kappa_{\Omega}^G(\calM[n] \langle m \rangle) = \kappa_{\Omega}^G(\calM)[-n+m]\langle -m \rangle$ and such that the following diagram commutes:
\[
\xymatrix@C=2cm@R=0.6cm{
\calD^{\mathrm{c}}_{G \times \Gm}(F_1 \, \rcap_E \, F_2) \ar[r]^{\kappa_{\Omega}^G}_{\sim} \ar[d]_-{\For} & \calD^{\mathrm{c}}_{G \times \Gm}(F_1^{\bot} \, \rcap_{E^*} \, F_2^{\bot})^\op \ar[d]^-{\For} \\
\calD^{\mathrm{c}}_{\Gm}(F_1 \, \rcap_E \, F_2) \ar[r]^{\kappa_{\Omega}}_{\sim} & \calD^{\mathrm{c}}_{\Gm}(F_1^{\bot} \, \rcap_{E^*} \, F_2^{\bot})^\op,
}
\]
where the vertical arrows are the forgetful functors.
\end{thm}

\begin{proof}
We use the same notation as in \cite{MR2}; in particular we will consider the equivalences $\overline{\scra}$, $\overline{\scra}^\bc$, $\mathbf{D}_\Omega^{\calT}$ and $\xi$ constructed in \cite[\S 1]{MR2}. With this notation we have $\kappa_\Omega = \xi \circ \overline{\scra}^\bc \circ \mathbf{D}_\Omega^{\calT}$.

It is straightforward to construct an equivalence of categories $\overline{\scra}_G$ which makes the following diagram commutative:
\[
\xymatrix@C=2cm@R=0.6cm{
\calD(\calT\Mod^G_-) \ar[r]^-{\overline{\scra}_G}_-{\sim} \ar[d]_-{\For} & \calD(\calS\Mod_-^G) \ar[d]^-{\For} \\
\calD(\calT\Mod_-) \ar[r]^-{\overline{\scra}}_-{\sim} & \calD(\calS\Mod_-^G).
}
\]
(Here, as in \cite{MR2}, the index ``$-$'' means the subcategories of dg-modules which are bounded above for the internal grading. A similar convention applies to the index ``$+$'' below.) This equivalence restricts to an equivalence
\[
\overline{\scra}_G^\bc : \calD^\bc(\calT\Mod^G_-) \xrightarrow{\sim} \calD^\bc(\calS\Mod_-^G)
\]
(where, as in \cite{MR2} or in \S\ref{ss:duality-dg-equiv}, the exponent ``$\bc$'' means the subcategories of dg-modules whose internal degree components have bounded and coherent cohomology).
Now in \S\ref{ss:duality-dg-equiv} we have constructed an equivalence $\mathrm{D}_\Omega^{\calT,G}$ which induces an equivalence
\[
\mathbf{D}_\Omega^{\calT,G} : \calD^\bc(\calT\Mod^G_+) \to \calD^\bc(\calT\Mod^G_-)^\op.
\]
Moreover, by Proposition~\ref{prop:duality-dg-equiv} the following diagram commutes:
\[
\xymatrix@C=2cm@R=0.6cm{
\calD^\bc(\calT\Mod_+^G) \ar[r]^-{\mathbf{D}_\Omega^{\calT,G}}_-{\sim} \ar[d]_-{\For} & \calD^\bc(\calT\Mod_-^G)^\op \ar[d]^-{\For} \\
\calD^\bc(\calT\Mod_+) \ar[r]^-{\mathbf{D}_\Omega^{\calT}}_-{\sim} & \calD^\bc(\calT\Mod_-)^\op.
}
\]
Finally the regrading functor has an obvious $G$-equivariant analogue $\xi^G$ and, setting $\kappa_\Omega^G:=\xi^G \circ \overline{\scra}_G^\bc \circ \mathbf{D}_\Omega^{\calT,G}$ we obtain an equivalence which makes the following diagram commutative:
\begin{equation}
\label{eqn:diagram-kappa-equiv}
\vcenter{
\xymatrix@C=2cm@R=0.6cm{
\calD^\bc(\calT\Mod^G_+) \ar[r]^{\kappa_{\Omega}^G}_{\sim} \ar[d]_-{\For} & \calD^\bc(\calR\Mod^G_+)^\op \ar[d]^-{\For} \\
\calD^\bc(\calT\Mod_+) \ar[r]^{\kappa_{\Omega}}_{\sim} & \calD^\bc(\calR\Mod_-)^\op.
}
}
\end{equation}

It is easy to check that the natural functors
\[
\calD^\fg(\calT\Mod_+^G) \to \calD^\fg(\calT\Mod^G) \quad \text{and} \quad \calD^\fg(\calR\Mod_-^G) \to \calD^\fg(\calR\Mod^G)
\]
are equivalences of categories and, using the commutativity of~\eqref{eqn:diagram-kappa-equiv}, that under these equivalences $\kappa_\Omega^G$ restricts to an equivalence
\[
\kappa_\Omega^G : \calD^\fg(\calT\Mod^G) \ \xrightarrow{\sim} \ \calD^\fg(\calR\Mod^G)^\op.
\]
This finishes the proof.
\end{proof}

\begin{remark}
From now on we will omit the exponent ``$G$'', and write $\kappa_\Omega$ instead of $\kappa_\Omega^G$. This convention is justified by the commutativity of the diagram in Theorem~\ref{thm:lkd-equivariant}.
\end{remark}

\subsection{Linear Koszul duality and morphisms of vector bundles}
\label{ss:morphisms-lkd}

We let $G$, $X$, $E$ be as in \S\ref{ss:lkd-equivariant}. Let also $E'$ be another $G$-equivariant vector bundle on $X$, and let
\[
\xymatrix@R=0.3cm{
E \ar[rd] \ar[rr]^{\phi} & & E' \ar[ld] \\
& X &
}
\]
be a morphism of $G$-equivariant vector bundles over $X$. We consider $G$-stable subbundles $F_1, F_2 \subseteq E$ and $F_1', F_2' \subseteq E'$, and assume that
\[
\phi(F_1) \subseteq F_1', \qquad \phi(F_2) \subseteq F_2'.
\]
Let $\calE, \, \calF_1, \, \calF_2, \, \calE', \, \calF_1', \, \calF_2'$ be the respective sheaves of sections of $E, \, F_1, \, F_2, \, E', \, F_1', \, F_2'$. We consider the ($G$-equivariant) complexes $\calX$ (for the vector bundle $E$) and $\calX'$ (for the vector bundle $E'$) defined as in \S\ref{ss:lkd-equivariant}, and the associated dg-algebras $\calT$, $\calR$, $\calT'$, $\calR'$.
By Theorem~\ref{thm:lkd-equivariant} we have linear Koszul duality equivalences
\begin{align*}
\kappa_{\Omega} : \calD^{\mathrm{c}}_{G \times \Gm}(F_1 \, \rcap_E \, F_2) \ & \xrightarrow{\sim} \ \calD^{\mathrm{c}}_{G \times \Gm}(F_1^{\bot} \, \rcap_{E^*} \, F_2^{\bot})^\op, \\
\kappa_{\Omega}' : \calD^{\mathrm{c}}_{G \times \Gm}(F_1' \, \rcap_{E'} \, F_2') \ & \xrightarrow{\sim} \ \calD^{\mathrm{c}}_{G \times \Gm}((F_1')^{\bot} \, \rcap_{(E')^*} \, (F_2')^{\bot})^\op.
\end{align*}

The morphism $\phi$ defines a morphism of complexes $\calX' \to \calX$, to which we can apply (equivariant analogues of) the constructions of \cite[\S 2.1]{MR2}. More geometrically, $\phi$ induces a morphism of dg-schemes $\Phi : F_1 \, \rcap_E \, F_2 \to F_1' \, \rcap_{E'} \, F_2'$, and we have a (derived) direct image functor
\[
R\Phi_* : \calD^{\mathrm{c}}_{G \times \Gm}(F_1 \, \rcap_E \, F_2) \to \calD(\calT'\Mod^G).
\]
(This functor is just the restriction of scalars functor associated with the morphism $\calT' \to \calT$.)

The following result immediately follows from \cite[Lemma 2.3.1]{MR2}. 

\begin{lem}
\label{lem:direct-image-Phi}
Assume that the induced morphism between non-derived intersections $F_1 \cap_E F_2 \to F_1' \cap_{E'} F_2'$ is proper. Then the functor $R\Phi_*$ sends $\calD^{\mathrm{c}}_{G \times \Gm}(F_1 \, \rcap_E \, F_2)$ into $\calD^{\mathrm{c}}_{G \times \Gm}(F_1' \, \rcap_{E'} \, F_2')$. 
\end{lem}

We also consider the (derived) inverse image functor
\[
L\Phi^* : \calD^{\mathrm{c}}_{G \times \Gm}(F_1' \, \rcap_{E'} \, F_2') \to \calD(\calT\Mod^G).
\]
(This functor is just the extension of scalars functor associated with the morphism $\calT' \to \calT$.)

The morphism $\phi$ induces a morphism of vector bundles
\[
\psi:= {}^{\mathrm{t}}\vspace{-2pt}\phi : (E')^* \to E^*,
\]
which satisfies $\psi((F_i')^{\bot}) \subset F_i^{\bot}$ for $i=1,2$. Hence the above constructions and results also apply to $\psi$. We use similar notation.

The following result is an equivariant analogue of \cite[Proposition 2.3.2]{MR2}. The same proof applies; we leave the details to the reader.

\begin{prop}
\label{prop:morphisms}
\begin{enumerate}
\item 
Assume that the morphism of schemes $F_1 \cap_E F_2 \to F_1' \cap_{E'} F_2'$ induced by $\phi$ is proper. Then $L\Psi^*$ sends $\calD^{\mathrm{c}}_{G \times \Gm}(F_1^{\bot} \, \rcap_{E^*} \, F_2^{\bot})$ into $\calD^{\mathrm{c}}_{G \times \Gm}((F_1')^{\bot} \, \rcap_{(E')^*} \, (F_2')^{\bot})$. Moreover, there exists a natural isomorphism of functors 
\[
L\Psi^* \circ \kappa_{\Omega} \ \cong \ \kappa'_{\Omega} \circ R\Phi_* \quad : \calD^{\mathrm{c}}_{G \times \Gm}(F_1 \, \rcap_{E} \, F_2) \to \calD^{\mathrm{c}}_{G \times \Gm}((F_1')^{\bot} \, \rcap_{(E')^*} \, (F_2')^{\bot})^\op.
\]
\item 
Assume that the morphism of schemes $(F_1')^{\bot} \cap_{(E')^*} (F_2')^{\bot} \to F_1^{\bot} \cap_{E'} F_2^{\bot}$ induced by $\psi$ is proper. Then $L\Phi^*$ sends $\calD^{\mathrm{c}}_{G \times \Gm}(F_1' \, \rcap_{E'} \, F_2')$ into $\calD^{\mathrm{c}}_{G \times \Gm}(F_1 \, \rcap_E \, F_2)$. Moreover, there exists a natural isomorphism of functors 
\[
\kappa_{\Omega} \circ L\Phi^* \ \cong \ R\Psi_* \circ \kappa'_{\Omega} \quad : \calD^{\mathrm{c}}_{G \times \Gm}(F_1' \, \rcap_{E'} \, F_2') \to \calD^{\mathrm{c}}_{G \times \Gm}(F_1^{\bot} \, \rcap_{E^*} \, F_2^{\bot})^\op.
\]
\end{enumerate}

In particular, if both assumptions are satisfied, the following diagram is commutative:
\[
\xymatrix@C=2cm{
\calD^{\mathrm{c}}_{G \times \Gm}(F_1 \, \rcap_E \, F_2) \ar[r]^-{\kappa_{\Omega}}_-{\sim} \ar@<0.5ex>[d]^-{R\Phi_*} & \calD^{\mathrm{c}}_{G \times \Gm}(F_1^{\bot} \, \rcap_{E^*} \, F_2^{\bot})^\op \ar@<0.5ex>[d]^-{L\Psi^*} \\
\calD^{\mathrm{c}}_{G \times \Gm}(F_1' \, \rcap_{E'} \, F_2') \ar@<0.5ex>[u]^-{L\Phi^*} \ar[r]^-{\kappa'_{\Omega}}_-{\sim} & \calD^{\mathrm{c}}_{G \times \Gm}((F_1')^{\bot} \, \rcap_{(E')^*} \, (F_2')^{\bot})^\op. \ar@<0.5ex>[u]^-{R\Psi_*}
}
\]
\end{prop}

\subsection{Particular case: inclusion of a subbundle}
\label{ss:subbundles}

We will mainly use only a special case of Proposition~\ref{prop:morphisms}, which we state here for future reference. It is the case when $E=E'$, $\phi=\Id$, $F_1=F_1'$ (and $F_2'$ is any $G$-stable subbundle containing $F_2$). In this case we denote by
\[
f: F_1 \, \rcap_E \, F_2 \to F_1 \, \rcap_E \, F_2', \qquad g: F_1^{\bot} \, \rcap_{E^*} \, (F_2')^{\bot} \to F_1^{\bot} \, \rcap_{E^*} \, F_2^{\bot}
\]
the morphisms of dg-schemes induced by $F_2 \hookrightarrow F_2'$, $(F_2')^{\bot} \hookrightarrow F_2^{\bot}$. The assumption that the morphisms between non-derived intersections are proper is always satisfied here (because these morphisms are closed embeddings). Hence by Proposition~\ref{prop:morphisms} we have functors
\[
Rf_* : \calD^{\mathrm{c}}_{G \times \Gm}(F_1 \, \rcap_E \, F_2) \  \to \ \calD^{\mathrm{c}}_{G \times \Gm}(F_1 \, \rcap_E \, F_2'), \qquad
Lf^* : \calD^{\mathrm{c}}_{G \times \Gm}(F_1 \, \rcap_E \, F_2') \ \to \ \calD^{\mathrm{c}}_{G \times \Gm}(F_1 \, \rcap_E \, F_2),
\]
and similarly for $g$. Moreover, the following proposition holds true.

\begin{prop}
\label{prop:inclusion}
Consider the following diagram:
\[
\xymatrix@C=2cm{
\calD^{\mathrm{c}}_{G \times \Gm}(F_1 \, \rcap_E \, F_2) \ar[r]^-{\kappa_{\Omega}}_-{\sim} \ar@<0.5ex>[d]^-{Rf_*} & \calD^{\mathrm{c}}_{G \times \Gm}(F_1^{\bot} \, \rcap_{E^*} \, F_2^{\bot})^\op \ar@<0.5ex>[d]^-{Lg^*} \\
\calD^{\mathrm{c}}_{G \times \Gm}(F_1 \, \rcap_E \, F_2') \ar[r]^-{\kappa'_{\Omega}}_-{\sim} \ar@<0.5ex>[u]^-{Lf^*} & \calD^{\mathrm{c}}_{G \times \Gm}(F_1^{\bot} \, \rcap_{E^*} \, (F_2')^{\bot})^\op. \ar@<0.5ex>[u]^-{Rg_*}
}
\]
There exist natural isomorphisms of functors
\[
\kappa'_{\Omega} \circ Rf_* \ \cong \ Lg^* \circ \kappa_{\Omega} \quad \text{ and } \quad \kappa_{\Omega} \circ Lf^* \ \cong \ Rg_* \circ \kappa'_{\Omega}.
\]
\end{prop}

\subsection{Linear Koszul duality and base change}
\label{ss:base-change-lkd}

Let $X$ and $Y$ be complex algebraic varieties, each endowed with an action of a reductive algebraic group $G$. We assume that condition~\eqref{eqn:enough-flats} holds on $X$ and $Y$, and we let $\Omega$ be an object of $\calD^b \Coh^G(Y)$ whose image in $\calD^b \Coh(Y)$ is a dualizing complex. We let $\pi: X \to Y$ be a $G$-equivariant morphism. Then $\pi^! \Omega$ is an object of $\calD^b \Coh^G(X)$ whose image in $\calD^b \Coh(X)$ is a dualizing complex.

Consider a $G$-equivariant vector bundle $E$ on $Y$, and let $F_1, F_2 \subset E$ be $G$-equivariant subbundles. Consider also $E^X:=E \times_Y X$, which is a $G$-equivariant vector bundle on $X$, and the subbundles $F_i^X:=F_i \times_Y X \subset E^X$ ($i=1,2$). If $\calE, \, \calF_1, \, \calF_2$ are the respective sheaves of sections of $E, \, F_1, \, F_2$, then $\pi^* \calE, \, \pi^* \calF_1, \, \pi^* \calF_2$ are the sheaves of sections of $E^X, \, F_1^X, \, F_2^X$, respectively. Out of these data we define the complexes $\calX_X$ and $\calX_Y$ as in \S\ref{ss:lkd-equivariant}, and then the dg-algebras $\calT_X$, $\calS_X$, $\calR_X$ and $\calT_Y$, $\calS_Y$, $\calR_Y$. Note that we have natural isomorphisms of dg-algebras
\[
\calT_X \cong \pi^* \calT_Y, \quad \calS_X \cong \pi^* \calS_Y, \quad \calR_X \cong \pi^* \calR_Y.
\]
We define the categories
\begin{align*}
\calD_{G \times \Gm}^{\mathrm{c}}(F_1 \, \rcap_E \, F_2), & \quad \calD_{G \times \Gm}^{\mathrm{c}}(F_1^{\bot} \, \rcap_{E^*} \, F_2^{\bot}) \\
\calD_{G \times \Gm}^{\mathrm{c}}(F_1^X \, \rcap_{E^X} \, F_2^X), & \quad \calD_{G \times \Gm}^{\mathrm{c}}((F_1^X)^{\bot} \, \rcap_{(E^X)^*} \, (F_2^X)^{\bot})
\end{align*}
as in \S\ref{ss:lkd-equivariant}. Then by Theorem~\ref{thm:lkd-equivariant} there are equivalences of categories
\begin{align*}
\kappa^X_{\pi^! \Omega} : \calD^{\mathrm{c}}_{G \times \Gm}(F_1^X \, \rcap_{E^X} \, F_2^X) \ & \xrightarrow{\sim} \ \calD^{\mathrm{c}}_{G \times \Gm}((F_1^X)^{\bot} \, \rcap_{(E^X)^*} \, (F_2^X)^{\bot})^\op, \\
\kappa^Y_\Omega : \calD^{\mathrm{c}}_{G \times \Gm}(F_1 \, \rcap_E \, F_2) \ & \xrightarrow{\sim} \ \calD^{\mathrm{c}}_{G \times \Gm}(F_1^{\bot} \, \rcap_{E^*} \, F_2^{\bot})^\op.
\end{align*}
If $X$ and $Y$ are smooth varieties, then $\Omega$ is a shift of a line bundle, so that $\pi^* \Omega$ is also a dualizing complex on $X$. Hence under this condition we also have an equivalence
\[
\kappa^X_{\pi^* \Omega} : \calD^{\mathrm{c}}_{G \times \Gm}(F_1^X \, \rcap_{E^X} \, F_2^X) \ \xrightarrow{\sim} \ \calD^{\mathrm{c}}_{G \times \Gm}((F_1^X)^{\bot} \, \rcap_{(E^X)^*} \, (F_2^X)^{\bot})^\op.
\]

The morphism of schemes $\pi$ induces a morphism of dg-schemes
\[
\hat{\pi} : F_1^X \, \rcap_{E^X} \, F_2^X \to F_1 \, \rcap_E \, F_2.
\]
This morphism can be represented by the natural morphism of dg-ringed spaces $(X, \, \calT_X) \to (Y, \, \calT_Y)$. We have derived functors $R\hat{\pi}_*$ and $L\hat{\pi}^*$ for this morphism by the constructions of \S\S\ref{ss:inverse-image-equiv}--\ref{ss:direct-image-equiv}. Note in particular that $\calD_{G \times \Gm}^{\mathrm{c}}(F_1^X \, \rcap_{E^X} \, F_2^X)$ is a subcategory of $\calD^+(\calT_X\Mod^G)$, so that $R\hat{\pi}_*$ is defined on this category.

As in \cite{MR2}, we will say that $\pi$ \emph{has finite Tor-dimension} if for any $\calF$ in $\QCoh(Y)$, the object $Lf^* \calF$ of $\calD\QCoh(X)$ has bounded cohomology.

\begin{lem}
\begin{enumerate}
\item
Assume $\pi$ has finite Tor-dimension. Then the functor
\[
L \hat{\pi}^* : \calD_{G \times \Gm}^{\mathrm{c}}(F_1 \, \rcap_E \, F_2) \ \to \ \calD(\calT_X\Mod^G)
\]
takes values in $\calD_{G \times \Gm}^{\mathrm{c}}(F_1^X \, \rcap_{E^X} \, F_2^X)$.
\item
Assume $\pi$ is proper. Then the functor
\[
R \hat{\pi}_* : \calD_{G \times \Gm}^{\mathrm{c}}(F_1^X \, \rcap_{E^X} \, F_2^X) \ \to \ \calD(\calT_Y \Mod^G)
\]
takes values in $\calD_{G \times \Gm}^{\mathrm{c}}(F_1 \, \rcap_E \, F_2)$.
\end{enumerate}
\end{lem}

\begin{proof}
Statement $(1)$ follows from its non-equivariant analogue (see \cite[Lemma 3.1.2]{MR2}) and the commutativity of diagram~\eqref{eqn:diagram-inverse-image}. The proof of $(2)$ is similar, using again \cite[Lemma 3.1.2]{MR2} and the commutativity of the diagram in Proposition~\ref{prop:direct-image-equiv}.
\end{proof}

Similarly, $\pi$ induces a morphism of dg-schemes
\[
\tilde{\pi} : (F_1^X)^{\bot} \, \rcap_{(E^X)^*} \, (F_2^X)^{\bot} \to F_1^{\bot} \, \rcap_{E^*} \, F_2^{\bot}
\]
hence, if $\pi$ has finite Tor-dimension, a functor
\[
L \tilde{\pi}^* : \calD_{G \times \Gm}^{\mathrm{c}}(F_1^{\bot} \, \rcap_{E^*} \, F_2^{\bot}) \ \to \ \calD_{G \times \Gm}^{\mathrm{c}}((F_1^X)^{\bot} \, \rcap_{(E^X)^*} \, (F_2^X)^{\bot}), \]
and, if $\pi$ is proper, a functor
\[
R \tilde{\pi}_* : \calD_{G \times \Gm}^{\mathrm{c}}((F_1^X)^{\bot} \, \rcap_{(E^X)^*} \, (F_2^X)^{\bot}) \ \to \ \calD_{G \times \Gm}^{\mathrm{c}}(F_1^{\bot} \, \rcap_{E^*} \, F_2^{\bot}).
\]

\begin{prop} 
\label{prop:basechange}
\begin{enumerate}
\item
If $\pi$ is proper, there exists a natural isomorphism of functors
\[
\kappa^Y_\Omega \circ R\hat{\pi}_* \ \cong \ R\tilde{\pi}_* \circ \kappa^X_{\pi^! \Omega} \quad : \calD_{G \times \Gm}^{\mathrm{c}}(F_1^X \, \rcap_{E^X} \, F_2^X) \to \calD_{G \times \Gm}^{\mathrm{c}}(F_1^{\bot} \, \rcap_{E^*} \, F_2^{\bot})^\op.
\]
\item
If $X$ and $Y$ are smooth varieties, then there exists a natural isomorphism of functors
\[
L\tilde{\pi}^* \circ \kappa^Y_\Omega \ \cong \ \kappa^X_{\pi^* \Omega} \circ L\hat{\pi}^* \quad : \calD_{G \times \Gm}^{\mathrm{c}}(F_1 \, \rcap_{E} \, F_2) \to \calD_{G \times \Gm}^{\mathrm{c}}((F_1^X)^{\bot} \, \rcap_{(E^X)^*} \, (F_2^X)^{\bot})^\op.
\]
\end{enumerate}
\end{prop}

\begin{proof}
$(1)$ 
In~\cite[Proposition 3.4.2]{MR2} we have proved a similar isomorphism in the non-equivariant setting. As in the proof of Corollary~\ref{cor:duality-equiv}, the equivariant case follows, using Lemma~\ref{lem:morphisms-dg-equiv}. (Here we also use the compatibility of the functors $\kappa^Y_\Omega$, $R\hat{\pi}_*$, $R\tilde{\pi}_*$ and $\kappa^X_{\pi^! \Omega}$ with their non-equivariant analogues, see Proposition~\ref{prop:direct-image-equiv} and Theorem~\ref{thm:lkd-equivariant}.)

$(2)$
The proof is similar.
\end{proof}

\section{Linear Koszul Duality and Convolution}
\label{sec:convolution}


From now on we will specialize to a particular geometric situation suitable to convolution algebras (and use slightly different notation). We fix a smooth and proper complex algebraic variety $X$, endowed with an action of a reductive algebraic group $G$. Note that condition~\eqref{eqn:enough-flats} is satisfied on any such variety by \cite[Proposition 5.1.26]{CG}.

\subsection{First description of convolution}
\label{ss:defconvolution}

Let $V$ be a finite dimensional $G$-module, and $F \subset E:= V \times X$ a $G$-equivariant subbundle of the trivial vector bundle with fiber $V$ over X. We denote by $\calF$ the sheaf of sections of $F$. Let $\Delta V \subset V \times V$ be the diagonal. We will apply the constructions of \S\ref{ss:lkd-equivariant} to the $G$-equivariant vector bundle $E \times E$ over $X \times X$ (for the diagonal $G$-action) and the $G$-stable subbundles $\Delta V \times X \times X$ and $F \times F$. Note that the derived intersection
\[
(\Delta V \times X \times X) \, \rcap_{E \times E} \, (F \times F)
\]
is quasi-isomorphic to the derived fiber product $F \, {\stackrel{_R}{\times}}_V \, F$ in the sense of \cite[\S 3.7]{BR}.

We want to define a convolution product on the category $
\calD^{\mathrm{c}}_{G \times \Gm} \bigl( (\Delta V \times X \times X) \, \rcap_{E \times E} \, (F \times F) \bigr)$.
More concretely, by definition (see \S\ref{ss:lkd-equivariant}) we have
\begin{equation}
\label{eqn:dg-alg}
\calD^{\mathrm{c}}_{G \times \Gm} \bigl( (\Delta V \times X \times X) \, \rcap_{E \times E} \, (F \times F) \bigr) \ \cong \ \calD^\fg \bigl( \mathrm{S}_{\calO_{X \times X}}(\calF^{\vee} \boxplus \calF^{\vee}) \otimes_{\C} (\bigwedge V^*)\Mod^G \bigr)
\end{equation}
where on the right hand side $V^*$ is identified with the orthogonal of $\Delta V$ in $V \times V$, \emph{i.e.}~with the anti-diagonal copy of $V^*$ in $V^* \times V^*$, and the differential is induced by the morphism $V^* \otimes_\C \calO_{X^2} \to \calF^{\vee} \boxplus \calF^{\vee}$ induced by the morphism $F \times F \hookrightarrow E \times E \twoheadrightarrow \bigl( (V \times V)/\Delta V \bigr) \times X^2$.

For $(i,j)=(1,2)$, $(2,3)$ or $(1,3)$ we have the projection $p_{i,j} : X^3 \to X^2$ on the $i$-th and $j$-th factors. There are associated morphisms of dg-schemes
\[
\widehat{p_{1,2}} : (\Delta V \times X^3) \, \rcap_{E \times E \times X} \, (F \times F \times X) \to (\Delta V \times X^2) \, \rcap_{E \times E} \, (F \times F),
\]
$\widehat{p_{2,3}}$, $\widehat{p_{1,3}}$, and functors $L(\widehat{p_{1,2}})^*$, $L(\widehat{p_{2,3}})^*$, $R(\widehat{p_{1,3}})_*$ (see \S\ref{ss:base-change-lkd}; in this setting $E \times E \times X$ is considered as a vector bundle over $X^3$). For $i=1,2,3$ we also denote by $p_i : X^3 \to X$ the projection on the $i$-th factor.

Next we consider a bifunctor
\begin{multline}
\label{eq:bifunctor}
\calC\bigl( \mathrm{S}_{\calO_{X^3}}(p_{1,2}^* (\calF^{\vee} \boxplus \calF^{\vee})) \otimes_{\C} ( \bigwedge V^*) \Mod^G \bigr) \times \calC \bigl( \mathrm{S}_{\calO_{X^3}}(p_{2,3}^* (\calF^{\vee} \boxplus \calF^{\vee})) \otimes_{\C} ( \bigwedge V^*) \Mod^G \bigr) \\
\to \ \calC \bigl( \mathrm{S}_{\calO_{X^3}}(p_{1,3}^* (\calF^{\vee} \boxplus \calF^{\vee})) \otimes_{\C} (\bigwedge V^*) \Mod^G \bigr).
\end{multline}
Here, in the first category the dg-algebra is the image under $p_{1,2}^*$ of the dg-algebra appearing in \eqref{eqn:dg-alg}, so that $\mathrm{S}_{\calO_{X^3}}(p_{1,2}^* (\calF^{\vee} \boxplus \calF^{\vee})) \otimes_{\C} ( \bigwedge V^*)$ is the structure sheaf of $(\Delta V \times X^3) \, \rcap_{E \times E \times X} \, (F \times F \times X)$. Similarly, the second category corresponds to the dg-scheme $(\Delta V \times X^3) \, \rcap_{X \times E \times E} \, (X \times F \times F)$, and the third one to the dg-scheme $(\Delta V \times X^3) \, \rcap_{E \times X \times E} \, (F \times X \times F)$. The bifunctor~\eqref{eq:bifunctor} takes the dg-modules $\calM_1$ and $\calM_2$ to the dg-module $\calM_1 \otimes_{\mathrm{S}_{\calO_{X^3}}(p_2^* \calF^{\vee})} \calM_2$, where the action of $\mathrm{S}_{\calO_{X^3}}(p_{1,3}^*(\calF^{\vee} \boxplus \calF^{\vee}))$ is the natural one (\emph{i.e.}~we forget the action of the middle copy of $\mathrm{S}_{\calO_X}(\calF^{\vee})$). To define the action of $\bigwedge V^*$, we remark that $\calM_1 \otimes_{\mathrm{S}_{\calO_{X^3}}(p_2^* \calF^{\vee})} \calM_2$ has a natural action of the dg-algebra $(\bigwedge V^*) \otimes_{\C} (\bigwedge V^*)$, which restricts to an action of $\bigwedge V^*$ via the algebra morphism $\bigwedge V^* \to (\bigwedge V^*) \otimes_{\C} (\bigwedge V^*)$ which sends an element $x \in V^*$ to $x \otimes 1 + 1 \otimes x$.

The bifunctor~\eqref{eq:bifunctor} has a derived bifunctor (which can be computed by means of $K$-flat resolutions, see Lemma \ref{lem:Kflats}), which induces a bifunctor denoted $(- \, \lotimes_{F^3} \, -)$:
\begin{multline*}
\calD^{\fg} \bigl( \mathrm{S}_{\calO_{X^3}}(p_{1,2}^*(\calF^{\vee} \boxplus \calF^{\vee})) \otimes_{\C} ( \bigwedge V^*) \Mod^G \bigr) \times \calD^\fg \bigl( \mathrm{S}_{\calO_{X^3}}(p_{2,3}^*( \calF^{\vee} \boxplus \calF^{\vee})) \otimes_{\C} ( \bigwedge V^*) \Mod^G \bigr) \\
\to \ \calD^\fg \bigl( \mathrm{S}_{\calO_{X^3}}(p_{1,3}^* (\calF^{\vee} \boxplus \calF^{\vee})) \otimes_{\C} (\bigwedge V^*) \Mod^G \bigr).
\end{multline*}
(This follows from the fact that the composition $F \times_V F \times_V F \hookrightarrow E \times E \times E \twoheadrightarrow E \times X \times E$ is proper, using arguments similar to those in the proof of \cite[Lemma 2.3.1]{MR2}; see also Lemma~\ref{lem:direct-image-Phi}).

Finally, we obtain a convolution product
\begin{multline*}
(- \star -) : \ \calD^{\mathrm{c}}_{G \times \Gm} \bigl( (\Delta V \times X \times X) \, \rcap_{E \times E} \, (F \times F) \bigr) \times \calD^{\mathrm{c}}_{G \times \Gm} \bigl( (\Delta V \times X \times X) \, \rcap_{E \times E} \, (F \times F) \bigr) \\
\to \ \calD^{\mathrm{c}}_{G \times \Gm} \bigl( (\Delta V \times X \times X) \, \rcap_{E \times E} \, (F \times F) \bigr)
\end{multline*}
defined by the formula
\[
\calM_1 \star \calM_2 \ := \ R(\widehat{p_{1,3}})_* \bigl( L(\widehat{p_{1,2}})^* \calM_2 \, \lotimes_{F^3} \, L(\widehat{p_{2,3}})^* \calM_1 \bigr).
\]
This convolution is associative in the natural sense. (We leave this verification to the reader; it will not be used in the paper.)

There is a natural $G \times \Gm$-equivariant ``projection''
\[
p: (\Delta V \times X \times X) \, \rcap_{E \times E} \, (F \times F) \to F \times F
\]
corresponding to the morphism of $\calO_{X^2}$-dg-algebras
\[
\mathrm{S}_{\calO_{X^2}}(\calF^{\vee} \boxplus \calF^{\vee}) \to \mathrm{S}_{\calO_{X^2}}(\calF^{\vee} \boxplus \calF^{\vee}) \otimes_{\C} ( \bigwedge V^*),
\]
and an associated direct image functor $Rp_*$. The essential image of $Rp_*$ lies in the full subcategory $\calD^b_{\mathrm{prop}} \Coh^G(F \times F)$ of $\calD^b \Coh^G(F \times F)$ whose objects are the complexes whose support is contained in a subvariety $Z \subset F \times F$ such that both projections $Z \to F$ are proper. This category $\calD^b_{\mathrm{prop}} \Coh^G(F \times F)$ has a natural convolution product (see \emph{e.g.}~\cite[\S 1.2]{R1}), and the functor
\[
Rp_* : \calD^{\mathrm{c}}_{G \times \Gm} \bigl( (\Delta V \times X \times X) \, \rcap_{E \times E} \, (F \times F) \bigr) \to \calD^b_{\mathrm{prop}} \Coh^G(F \times F)
\]
is compatible with the two convolution products.

\subsection{Alternative description}
\label{ss:defconvolution2}

Before studying the compatibility of convolution with linear Koszul duality we give an alternative description of the convolution bifunctor, in terms of functors considered in Section \ref{sec:LKD}. Consider the morphism
\[
i : \left\{
\begin{array}{ccc}
X^3 & \hookrightarrow & X^4 \\
(x,y,z) & \mapsto & (x,y,y,z)
\end{array}
\right. ,
\]
and the vector bundle $E^4$ over $X^4$. In \S\ref{ss:base-change-lkd} we have defined a ``base change'' functor
\begin{multline*}
L\hat{i}^* : \calD^{\mathrm{c}}_{G \times \Gm} \bigl( (\Delta V \times \Delta V \times X^4) \, \rcap_{E^4} \, F^4 \bigr)
\\ \to \ \calD^{\mathrm{c}}_{G \times \Gm} \bigl( (\Delta V \times \Delta V \times X^3) \, \rcap_{E \times (E \times_X E) \times E} \, (F \times (F \times_X F) \times F) \bigr).
\end{multline*} 
Next, consider the inclusion of vector subbundles of $E \times (E \times_X E) \times E$ (over $X^3$)
\[
F \times F^{\mathrm{diag}} \times F \ \hookrightarrow \ F \times (F \times_X F) \times F,
\]
where $F^{\mathrm{diag}} \subset F \times_X F$ is the diagonal copy of $F$. In \S\ref{ss:subbundles} we have defined a functor
\begin{multline*}
Lf^* : \calD^{\mathrm{c}}_{G \times \Gm} \bigl( (\Delta V \times \Delta V \times X^3) \, \rcap_{E \times (E \times_X E) \times E} \, (F \times (F \times_X F) \times F) \bigr) \\
\to \ \calD^{\mathrm{c}}_{G \times \Gm} \bigl( (\Delta V \times \Delta V \times X^3) \, \rcap_{E \times (E \times_X E) \times E} \, (F \times F^{\mathrm{diag}} \times F) \bigr).
\end{multline*} 
Consider then the morphisms of vector bundles over $X^3$
\[
\phi_1 :  E^3 \cong V^3 \times X^3 \to E \times X \times E \cong V^2 \times X^3, \quad \phi_2 : E^3 \cong V^3 \times X^3 \ \to \ E \times (E \times_X E) \times E \cong V^4 \times X^3,
\]
induced by the linear maps
\[
\varphi_1 : 
\left\{
\begin{array}{ccc}
V^3 & \to & V^2 \\
(a,b,c) & \mapsto & (a,c)
\end{array}
\right. \quad \text{and} \quad \varphi_2 : 
\left\{
\begin{array}{ccc}
V^3 & \to & V^4 \\
(a,b,c) & \mapsto & (a,b,b,c)
\end{array}
\right. .
\]
Using the constructions of \S\ref{ss:morphisms-lkd} we have associated functors
\begin{gather*}
R(\Phi_1)_* : \calD^{\mathrm{c}}_{G \times \Gm} \bigl( (\Delta^3 V \times X^3) \, \rcap_{E^3} \, F^3 \bigr) \
\to \ \calD^{\mathrm{c}}_{G \times \Gm} \bigl( (\Delta V \times X^3) \, \rcap_{E \times X \times E} \, (F \times X \times F) \bigr), \\
L(\Phi_2)^* : \calD^{\mathrm{c}}_{G \times \Gm} \bigl( (\Delta V \times \Delta V \times X^3) \, \rcap_{E \times (E \times_X E) \times E} \, (F \times F^{\mathrm{diag}} \times F) \bigr) \ \to \ \calD^{\mathrm{c}}_{G \times \Gm} \bigl( (\Delta^3 V \times X^3) \, \rcap_{E^3} \, F^3 \bigr).
\end{gather*}
(Here $\Delta^3 V \subset V^3$ is the diagonal copy of $V$.)
Finally, note that if $\calM_1$, $\calM_2$ are two objects of $\calD^{\mathrm{c}}_{G \times \Gm} \bigl( (\Delta V \times X^2) \, \rcap_{E \times E} \, (F \times F) \bigr)$, then the external tensor product $\calM_2 \boxtimes \calM_1$ is naturally an object of the category $\calD^{\mathrm{c}}_{G \times \Gm} \bigl( (\Delta V \times \Delta V \times X^4) \, \rcap_{E^4} \, F^4 \bigr)$. 

\begin{lem}
For $\calM_1$, $\calM_2$ in $\calD^{\mathrm{c}}_{G \times \Gm} \bigl( (\Delta V \times X^2) \, \rcap_{E \times E} \, (F \times F) \bigr)$, there exists a bifunctorial isomorphism
\begin{equation}
\label{eq:defconvolution2}
\calM_1 \star \calM_2 \ \cong \ R(\widehat{p_{1,3}})_* \circ R(\Phi_1)_* \circ L(\Phi_2)^* \circ Lf^* \circ L\hat{i}^* (\calM_2 \boxtimes \calM_1)
\end{equation}
in $\calD^{\mathrm{c}}_{G \times \Gm} \bigl( (\Delta V \times X^2) \, \rcap_{E \times E} \, (F \times F) \bigr)$.
\end{lem}

\begin{proof}
What we have to construct is a natural isomorphism
\begin{equation}
\label{eqn:isom-bifunctors}
L(\widehat{p_{1,2}})^* \calM_2 \, \lotimes_{F^3} \, L(\widehat{p_{2,3}})^* \calM_1 \ \cong \ R(\Phi_1)_* \circ L(\Phi_2)^* \circ Lf^* \circ L\hat{i}^* (\calM_2 \boxtimes \calM_1).
\end{equation}


Note that the morphism of dg-schemes $\Phi_1$ is associated with the natural morphism of dg-algebras $\mathrm{S}_{\calO_{X^3}}(p_{1,3}^* (\calF^{\vee} \boxplus \calF^{\vee})) \otimes_{\C} (\bigwedge V^*) \to  \mathrm{S}_{\calO_{X^3}}(\calF^{\vee} \boxplus \calF^{\vee} \boxplus \calF^{\vee}) \otimes_{\C} (\bigwedge V^*) \otimes_{\C} (\bigwedge V^*)$ on $X^3$. By construction there exists a bifunctor $(- \underline{\otimes}_{F^3} -)$ taking values in $\calC \bigl( \mathrm{S}_{\calO_{X^3}}(\calF^{\vee} \boxplus \calF^{\vee} \boxplus \calF^{\vee}) \otimes_{\C} (\bigwedge V^*) \otimes_{\C} (\bigwedge V^*) \Mod^G \bigr)$ and such that $(- \otimes_{F^3} -) = (\Phi_1)_* \circ (- \underline{\otimes}_{F^3} -)$. As the functor $(\Phi_1)_*$ is exact, passing to derived functors (and restricting to the appropriate subcategories) we obtain an isomorphism $(- \lotimes_{F^3} -) = R(\Phi_1)_* \circ (- \underline{\lotimes}_{F^3} -)$.

Let us also observe that, by standard properties of composition of left derived functors the composition $(\calM_1,\calM_2) \mapsto L(\widehat{p_{1,2}})^* \calM_2 \underline{\lotimes}_{F^3}  L(\widehat{p_{2,3}})^* \calM_1$ is the derived bifunctor of the corresponding composition $(\calM_1,\calM_2) \mapsto (\widehat{p_{1,2}})^* \calM_2 \underline{\otimes}_{F^3}  (\widehat{p_{2,3}})^* \calM_1$ of non-derived functors. The latter composition can also be described as the composition of the bifunctor $(\calM_1,\calM_2) \mapsto \calM_2 \boxtimes \calM_1$ with the (non-derived) inverse image under the morphism of dg-schemes
\[
\Upsilon : (\Delta^3 V \times X^3) \, \rcap_{E^3} \, F^3 \to (\Delta V \times \Delta V \times X^4) \, \rcap_{E^4} \, F^4
\]
whose underlying morphism of schemes is $i$ and whose underlying morphism of dg-algebras is the morphism
\begin{multline*}
\Bigl(\mathrm{S}(\calF^{\vee}) \boxtimes  \bigl( \mathrm{S}(\calF^{\vee}) \otimes_{\calO_X} \mathrm{S}(\calF^{\vee}) \bigr) \boxtimes \mathrm{S}(\calF^{\vee}) \Bigr) \otimes_{\C} (\bigwedge V^*) \otimes_{\C} (\bigwedge V^*) \to \\ 
\Bigl(\mathrm{S}(\calF^{\vee}) \boxtimes \mathrm{S}(\calF^{\vee}) \boxtimes \mathrm{S}(\calF^{\vee}) \Bigr) \otimes_{\C} (\bigwedge V^*) \otimes_{\C} (\bigwedge V^*)
\end{multline*}
induced by multiplication of the second and third copies of $\mathrm{S}(\calF^{\vee})$. By the same arguments as before, passing to derived functors we obtain an isomorphism $L(\widehat{p_{1,2}})^* \calM_2 \underline{\lotimes}_{F^3}  L(\widehat{p_{2,3}})^* \calM_1 \cong L\Upsilon^*(\calM_2 \boxtimes \calM_1)$. Summarizing, we have obtained a bifunctorial isomorphism
\begin{equation*}
L(\widehat{p_{1,2}})^* \calM_2 \, \lotimes_{F^3} \, L(\widehat{p_{2,3}})^* \calM_1 \ \cong \ R(\Phi_1)_* \circ L\Upsilon^*(\calM_2 \boxtimes \calM_1).
\end{equation*}
Comparing with \eqref{eqn:isom-bifunctors}, we see that it is sufficient to construct an isomorphism of functors
\begin{equation*}
L(\Phi_2)^* \circ Lf^* \circ L\hat{i}^* \ \cong \ L\Upsilon^*.
\end{equation*}
The latter isomorphism follows from the observation that $\Upsilon = \hat{i} \circ f \circ \Phi_2$, and standard results on composition of left derived functors (as e.g.~in \cite[Corollary 3.3.8]{BR}).
\end{proof}

\begin{remark}
\label{rk:qis}
The morphism $\Phi_2 : (\Delta^3 V \times X^3) \, \rcap_{E^3} \, F^3 \to (\Delta V \times \Delta V \times X^3) \, \rcap_{E \times (E \times_X E) \times E} \, (F \times F^{\mathrm{diag}} \times F)$ is a quasi-isomorphism of dg-schemes; in fact the associated morphism of dg-algebras as in \S\ref{ss:subbundles} is even an isomorphism. In particular, $L(\Phi_2)^*$ is an equivalence of categories (see \cite[Proposition 1.3.2]{MR}).
\end{remark}

\subsection{Compatibility with Koszul duality}
\label{ss:compatibilityconvolution}

Consider the situation of \S\S\ref{ss:defconvolution}--\ref{ss:defconvolution2}. We denote by $d_X$ the dimension of $X$, and by $\omega_X$ the canonical line bundle on $X$.

The orthogonal of $F \times F$ in $E \times E$ is $F^{\bot} \times
F^{\bot}$. On the other hand, the orthogonal of $\Delta V \times X^2$
in $E \times E$ is the anti-diagonal $\overline{\Delta} V^* \times X^2 \subset
E^* \times E^*$. There is an automorphism of $E \times E$ sending
$\overline{\Delta} V^* \times X^2$ to $\Delta V^* \times X^2$ and preserving $F^{\bot} \times F^{\bot}$, namely multiplication by $-1$ on the second copy
of $V^*$. Hence composing the linear Koszul duality equivalence of Theorem~\ref{thm:lkd-equivariant}
\[
\kappa : \calD^{\mathrm{c}}_{G \times \Gm} \bigl( (\Delta V \times X
\times X) \, \rcap_{E \times E} \, (F \times F) \bigr)
\ \xrightarrow{\sim} \ \calD^{\mathrm{c}}_{G \times \Gm} \bigl( (\overline{\Delta}
V^* \times X \times X) \, \rcap_{E^* \times E^*} \, (F^{\bot} \times
F^{\bot}) \bigr)^\op
\]
associated with the dualizing complex $\omega_X \boxtimes \calO_X [d_X]$
with the equivalence
\[
\Xi : \calD^{\mathrm{c}}_{G \times \Gm} \bigl( (\overline{\Delta}
V^* \times X \times X) \, \rcap_{E^* \times E^*} \, (F^{\bot} \times
F^{\bot}) \bigr) 
\xrightarrow{\sim} \ \calD^{\mathrm{c}}_{G \times \Gm} \bigl( (\Delta
V^* \times X \times X) \, \rcap_{E^* \times E^*} \, (F^{\bot} \times
F^{\bot}) \bigr)
\]
provides an equivalence
\[
\frakK : \calD^{\mathrm{c}}_{G \times \Gm} \bigl( (\Delta V \times X
\times X) \, \rcap_{E \times E} \, (F \times F) \bigr) \
\xrightarrow{\sim} \ \calD^{\mathrm{c}}_{G \times \Gm} \bigl( (\Delta
V^* \times X \times X) \, \rcap_{E^* \times E^*} \, (F^{\bot} \times
F^{\bot}) \bigr)^\op.
\]
The domain and the codomain of $\frakK$ are both endowed with a convolution product $\star$.

The main result of this section is the following proposition, whose proof is based on the results of \S\S\ref{ss:morphisms-lkd}--\ref{ss:base-change-lkd}.

\begin{prop}
\label{prop:convolution}
The equivalence $\frakK$ is compatible with convolution, \emph{i.e.}~for any
objects $\calM_1$, $\calM_2$ of $\calD^{\mathrm{c}}_{G \times \Gm}
\bigl( (\Delta V \times X \times X) \, \rcap_{E \times E} \, (F \times
F) \bigr)$ there exists a bifunctorial isomorphism
\[
\frakK(\calM_1 \star \calM_2) \
\cong \ \frakK(\calM_1) \star \frakK(\calM_2)
\]
in $\calD^{\mathrm{c}}_{G \times
  \Gm} \bigl( (\Delta V^* \times X \times X) \, \rcap_{E^* \times E^*}
\, (F^{\bot} \times F^{\bot}) \bigr)$.
\end{prop}

\begin{proof}
To compute the left hand side we use isomorphism~\eqref{eq:defconvolution2}. First, as in \S\ref{ss:defconvolution} we consider the projection $p_{1,3} : X^3 \to X^2$. In \S\ref{ss:base-change-lkd} we have defined functors
\[
R\widehat{p_{1,3}}_* : \calD^{\mathrm{c}}_{G \times \Gm} \bigl( (\Delta V \times X^3) \, \rcap_{E \times X \times E} \, (F \times X \times F) \bigr) \ \to \ \calD^{\mathrm{c}}_{G \times \Gm} \bigl( (\Delta V \times X^2) \, \rcap_{E \times E} \, (F \times F) \bigr),
\]
\vspace{-18pt}
\[
R\widetilde{p_{1,3}}_* : \calD^{\mathrm{c}}_{G \times \Gm} \bigl( (\overline{\Delta} V^* \times X^3) \, \rcap_{E^* \times X \times E^*} \, (F^{\bot} \times X \times F^{\bot}) \bigr) 
\ \to \ \calD^{\mathrm{c}}_{G \times \Gm} \bigl( (\overline{\Delta} V^* \times X^2) \, \rcap_{E^* \times E^*} \, (F^{\bot} \times F^{\bot}) \bigr).
\]
We denote by
\[
\kappa_{1,3} : \calD^{\mathrm{c}}_{G \times \Gm} \bigl( (\Delta V \times X^3) \, \rcap_{E \times X \times E} \, (F \times X \times F) \bigr) 
\ \xrightarrow{\sim} \ \calD^{\mathrm{c}}_{G \times \Gm} \bigl( (\overline{\Delta} V^* \times X^3) \, \rcap_{E^* \times X \times E^*} \, (F^{\bot} \times X \times F^{\bot}) \bigr)^\op
\]
the linear Koszul duality equivalence of Theorem~\ref{thm:lkd-equivariant} associated with the dualizing complex $(p_{1,3})^!(\omega_X \boxtimes \calO_X[d_X]) \cong \omega_X \boxtimes \omega_X \boxtimes \calO_X [2d_X]$. By Proposition~\ref{prop:basechange} we have an isomorphism of functors 
\begin{equation}
\label{eq:isofunctors1}
\kappa \circ R\widehat{p_{1,3}}_* \ \cong \ R\widetilde{p_{1,3}}_* \circ \kappa_{1,3}.
\end{equation}

Next consider, as in \S\ref{ss:defconvolution2}, the inclusion $i : X^3 \hookrightarrow X^4$. In addition to the functor $L\hat{i}^*$, consider
\begin{multline*}
L\tilde{i}^* : \calD^{\mathrm{c}}_{G \times \Gm} \bigl( (\overline{\Delta} V^* \times \overline{\Delta} V^* \times X^4) \, \rcap_{(E^*)^4} \, (F^{\bot})^4 \bigr) \\
\to \ \calD^{\mathrm{c}}_{G \times \Gm} \bigl( (\overline{\Delta} V^* \times \overline{\Delta} V^* \times X^3) \, \rcap_{E^* \times (E^* \times_X E^*) \times E^*} \, (F^{\bot} \times (F^{\bot} \times_X F^{\bot}) \times F^{\bot}) \bigr).
\end{multline*}
We denote by
\[
\kappa_4 : \calD^{\mathrm{c}}_{G \times \Gm} \bigl( (\Delta V \times \Delta V \times X^4) \, \rcap_{E^4} \, F^4 \bigr) \ \xrightarrow{\sim} \ \calD^{\mathrm{c}}_{G \times \Gm} \bigl( (\overline{\Delta} V^* \times \overline{\Delta} V^* \times X^4) \, \rcap_{(E^*)^4} \, (F^{\bot})^4 \bigr)^\op
\]
the linear Koszul duality equivalence associated with the dualizing complex $\omega_X \boxtimes \calO_X \boxtimes \omega_X \boxtimes \calO_X[2d_X]$ on $X^4$, and by
\begin{multline*}
\kappa_3 : \calD^{\mathrm{c}}_{G \times \Gm} \bigl( (\Delta V \times \Delta V \times X^3) \, \rcap_{E \times (E \times_X E) \times E} \, (F \times (F \times_X F) \times F) \bigr) \\ \xrightarrow{\sim} \ \calD^{\mathrm{c}}_{G \times \Gm} \bigl( (\overline{\Delta} V^* \times \overline{\Delta} V^* \times X^3) \, \rcap_{E^* \times (E^* \times_X E^*) \times E^*} \, (F^{\bot} \times (F^{\bot} \times_X F^{\bot}) \times F^{\bot}) \bigr)
\end{multline*}
the linear Koszul duality equivalence associated with the dualizing complex $\omega_X \boxtimes \omega_X \boxtimes \calO_X[2d_X]$ on $X^3$. By Proposition~\ref{prop:basechange} we have an isomorphism of functors
\begin{equation}
\label{eq:isofunctors2}
L\tilde{i}^* \circ \kappa_4 \ \cong \ \kappa_3 \circ L\hat{i}^*.
\end{equation}

As in \S\ref{ss:defconvolution2} again, consider now the inclusion $F \times F^{\mathrm{diag}} \times F \hookrightarrow F \times (F \times_X F) \times F$, and the induced morphisms of dg-schemes
\begin{multline*}
f : (\Delta V \times \Delta V \times X^3) \, \rcap_{E \times (E \times_X E) \times E} \, (F \times F^{\mathrm{diag}} \times F) \ \to \\
(\Delta V \times \Delta V \times X^3) \, \rcap_{E \times (E \times_X E) \times E} \, (F \times (F \times_X F) \times F),
\end{multline*}
\vspace{-20pt}
\begin{multline*}
g : (\overline{\Delta} V^* \times \overline{\Delta} V^* \times X^3) \, \rcap_{E^* \times (E^* \times_X E^*) \times E^*} \, (F^{\bot} \times (F^{\bot} \times_X F^{\bot}) \times F^{\bot}) \ \to \\
(\overline{\Delta} V^* \times \overline{\Delta} V^* \times X^3) \, \rcap_{E^* \times (E^* \times_X E^*) \times E^*} \, (F^{\bot} \times (F^{\mathrm{diag}})^{\bot} \times F^{\bot}).
\end{multline*}
In addition to the functor $Lf^*$, consider the functor
\begin{multline*}
Rg_* : \calD^{\mathrm{c}}_{G \times \Gm} \bigl( (\overline{\Delta} V^* \times \overline{\Delta} V^* \times X^3) \, \rcap_{E^* \times (E^* \times_X E^*) \times E^*} \, (F^{\bot} \times (F^{\bot} \times_X F^{\bot}) \times F^{\bot}) \bigr) \\
\to \ \calD^{\mathrm{c}}_{G \times \Gm} \bigl( (\overline{\Delta} V^* \times \overline{\Delta} V^* \times X^3) \, \rcap_{E^* \times (E^* \times_X E^*) \times E^*} \, (F^{\bot} \times (F^{\mathrm{diag}})^{\bot} \times F^{\bot}) \bigr)
\end{multline*}
defined as in \S\ref{ss:subbundles}. We denote by
\begin{multline*}
\kappa_3' : \calD^{\mathrm{c}}_{G \times \Gm} \bigl( (\Delta V \times \Delta V \times X^3) \, \rcap_{E \times (E \times_X E) \times E} \, (F \times F^{\mathrm{diag}} \times F) \bigr) \\ \xrightarrow{\sim} \ \calD^{\mathrm{c}}_{G \times \Gm} \bigl( (\overline{\Delta} V^* \times \overline{\Delta} V^* \times X^3) \, \rcap_{E^* \times (E^* \times_X E^*) \times E^*} \, (F^{\bot} \times (F^{\mathrm{diag}})^{\bot} \times F^{\bot}) \bigr)^\op
\end{multline*}
the linear Koszul duality equivalence associated with the dualizing complex $\omega_X \boxtimes \omega_X \boxtimes \calO_X[2d_X]$. Then by Proposition~\ref{prop:inclusion} we have an isomorphism of functors
\begin{equation}
\label{eq:isofunctors3}
\kappa_3' \circ Lf^* \ \cong \ Rg_* \circ \kappa_3.
\end{equation}

Finally, we denote by
\[
\kappa_3'' : \calD^{\mathrm{c}}_{G \times \Gm} \bigl( (\Delta^3 V \times X^3) \, \rcap_{E^3} \, F^3 \bigr) \ \xrightarrow{\sim} \ \calD^{\mathrm{c}}_{G \times \Gm} \bigl( ((\Delta^3 V)^\bot \times X^3) \, \rcap_{(E^*)^3} \, (F^{\bot})^3 \bigr)^\op
\]
the linear Koszul duality equivalence associated with the dualizing complex $\omega_X \boxtimes \omega_X \boxtimes \calO_X[2d_X]$. Consider the morphisms of vector bundles $\phi_1$, $\phi_2$
defined in \S\ref{ss:defconvolution2}. By Proposition~\ref{prop:morphisms}, the dual morphisms $\psi_1:={}^{\mathrm{t}}\vspace{-2pt}\phi_1$ and $\psi_2:={}^{\mathrm{t}}\vspace{-2pt}\phi_2$ induce functors
\[
L(\Psi_1)^* : \calD^{\mathrm{c}}_{G \times \Gm} \bigl( ((\Delta^3 V)^\bot \times X^3) \, \rcap_{(E^*)^3} \, (F^{\bot})^3 \bigr) \to \calD^{\mathrm{c}}_{G \times \Gm} \bigl( (\overline{\Delta} V^* \times X^3) \, \rcap_{E^* \times X \times E^*} \, (F^{\bot} \times X \times F^{\bot}) \bigr)
\]
and
\begin{multline*}
R(\Psi_2)_* : \calD^{\mathrm{c}}_{G \times \Gm} \bigl( (\overline{\Delta} V^* \times \overline{\Delta} V^* \times X^3) \, \rcap_{E^* \times (E^* \times_X E^*) \times E^*} \, (F^{\bot} \times (F^{\mathrm{diag}})^{\bot} \times F^{\bot}) \bigr) \\
\to \ \calD^{\mathrm{c}}_{G \times \Gm} \bigl( ((\Delta^3 V)^\bot \times X^3) \, \rcap_{(E^*)^3} \, (F^{\bot})^3 \bigr),
\end{multline*}
and we have isomorphisms of functors
\begin{equation}
\label{eq:isofunctors4}
\kappa_{1,3} \circ R(\Phi_1)_* \ \cong \ L(\Psi_1)^* \circ \kappa_3'' \quad \text{and} \quad \kappa_3'' \circ L(\Phi_2)^* \ \cong \ R(\Psi_2)_* \circ \kappa_3'.
\end{equation}

Combining isomorphisms~\eqref{eq:defconvolution2}, \eqref{eq:isofunctors1}, \eqref{eq:isofunctors2}, \eqref{eq:isofunctors3} and~\eqref{eq:isofunctors4} we obtain, for $\calM_1$ and $\calM_2$ in $\calD^{\mathrm{c}}_{G \times \Gm} \bigl( (\Delta V \times X^2) \, \rcap_{E \times E} \, (F \times F) \bigr)$, bifunctorial isomorphisms
\begin{align*}
\kappa(\calM_1 \star \calM_2) \ & \cong \ \kappa \circ R(\widehat{p_{1,3}})_* \circ R(\Phi_1)_* \circ L(\Phi_2)^* \circ Lf^* \circ L\hat{i}^* (\calM_2 \boxtimes \calM_1) \\
& \cong \ R(\widetilde{p_{1,3}})_* \circ L(\Psi_1)^* \circ R(\Psi_2)_* \circ Rg_* \circ L\tilde{i}^* \circ \kappa_4 (\calM_2 \boxtimes \calM_1). \end{align*}
It is clear by definition that $\kappa_4 (\calM_2 \boxtimes \calM_1) \cong \kappa(\calM_2) \boxtimes \kappa(\calM_1)$ in $\calD^{\mathrm{c}}_{G \times \Gm} \bigl( (\overline{\Delta} V^* \times \overline{\Delta} V^* \times X^4) \, \rcap_{(E^*)^4} \, (F^{\bot})^4 \bigr)$. Hence, using again isomorphism~\eqref{eq:defconvolution2} (but on the dual side) and taking the equivalences $\Xi$ into account, to finish the proof we only have to construct an isomorphism of functors
\begin{equation}
\label{eqn:isom-convolution-LKD}
L(\Psi_1)^* \circ R(\Psi_2)_* \circ Rg_* \ \cong \ R(\Phi'_1)_* \circ L(\Phi'_2)^*\circ L(f')^*
\end{equation}
where the functors are associated with the morphisms of dg-schemes defined as follows. (The functors are well defined thanks to the results of \S\S\ref{ss:morphisms-lkd}--\ref{ss:base-change-lkd}.) The morphism
\begin{multline*}
f' : (\overline{\Delta} V^* \times \overline{\Delta} V^* \times X^3) \, \rcap_{E^* \times (E^* \times_X E^*) \times E^*} \, (F^{\bot} \times (F^{\bot})^{\mathrm{adiag}} \times F^{\bot}) \\
\to (\overline{\Delta} V^* \times \overline{\Delta} V^* \times X^3) \, \rcap_{E^* \times (E^* \times_X E^*) \times E^*} \, (F^{\bot} \times (F^{\bot} \times_X F^{\bot}) \times F^{\bot})
\end{multline*}
is induced by the inclusion of vector subbundles
\[
F^{\bot} \times (F^{\bot})^{\mathrm{adiag}} \times F^{\bot} \hookrightarrow F^{\bot} \times (F^{\bot} \times_X F^{\bot}) \times F^{\bot}
\]
where $(F^{\bot})^{\mathrm{adiag}} \subset E^* \times_X E^*$ is the antidiagonal copy of $F^\bot$. Next, the morphism
\[
\Phi'_2 : (\Delta^3 V^* \times X^3) \, \rcap_{(E^*)^3} \, (F^\bot)^3 \to (\overline{\Delta} V^* \times \overline{\Delta} V^* \times X^3) \, \rcap_{E^* \times (E^* \times_X E^*) \times E^*} \, (F^{\bot} \times (F^{\bot})^{\mathrm{adiag}} \times F^{\bot})
\]
is induced by the linear map $\varphi_2' : (V^*)^3 \to (V^*)^4$ sending $(a,b,c)$ to $(a,-b,b,-c)$. Finally, the morphism
\[
\Phi'_1 : (\Delta^3 V^* \times X^3) \, \rcap_{(E^*)^3} \, (F^\bot)^3 \to (\overline{\Delta} V^* \times X^3) \, \rcap_{E^* \times X \times E^*} (F^\bot \times X \times F^\bot)
\]
is induced by the linear map $\varphi'_1 : (V^*)^3 \to (V^*)^2$ sending $(a,b,c)$ to $(a,-c)$.

To prove isomorphism \eqref{eqn:isom-convolution-LKD} we introduce three more morphisms of dg-schemes. First, the morphism
\[
h : ((\Delta^3 V)^\bot \times X^3) \, \rcap_{(E^*)^3} (F^\bot \times X \times F^\bot) \to ((\Delta^3 V)^\bot \times X^3) \, \rcap_{(E^*)^3} \, (F^{\bot})^3
\]
is induced by the inclusion of vector subbundles $F^\bot \times X \times F^\bot \hookrightarrow (F^{\bot})^3$. (Here $X$ is identified with the zero-section of $E^*$.) Next, the morphism
\[
\Omega : (\overline{\Delta} V^* \times X^3) \, \rcap_{E^* \times X \times E^*} (F^\bot \times X \times F^\bot) \to ((\Delta^3 V)^\bot \times X^3) \, \rcap_{(E^*)^3} (F^\bot \times X \times F^\bot)
\]
is induced by the morphism of vector bundles $\psi_1$. Finally, the morphism
\[
\Theta : (\overline{\Delta} V^* \times \overline{\Delta} V^* \times X^3) \, \rcap_{E^* \times (E^* \times_X E^*) \times E^*} \, (F^{\bot} \times (F^{\bot})^{\mathrm{adiag}} \times F^{\bot}) \to ((\Delta^3 V)^\bot \times X^3) \, \rcap_{(E^*)^3} (F^\bot \times X \times F^\bot)
\]
is induced by the morphism of vector bundles $\psi_2$.

Consider the following diagram:
\[
{\footnotesize
\xymatrix{
 (\overline{\Delta} V^* \times \overline{\Delta} V^* \times X^3) \, \rcap_{E^* \times (E^* \times_X E^*) \times E^*} \, (F^{\bot} \times (F^{\bot} \times_X F^{\bot}) \times F^{\bot}) \ar[r]^-{\Psi_2 \circ g} & ((\Delta^3 V)^\bot \times X^3) \, \rcap_{(E^*)^3} \, (F^{\bot})^3 \\
(\overline{\Delta} V^* \times \overline{\Delta} V^* \times X^3) \, \rcap_{E^* \times (E^* \times_X E^*) \times E^*} \, (F^{\bot} \times (F^{\bot})^{\mathrm{adiag}} \times F^{\bot}) \ar[u]^-{f'} \ar[r]^-{\Theta} & ((\Delta^3 V)^\bot \times X^3) \, \rcap_{(E^*)^3} (F^\bot \times X \times F^\bot) \ar[u]_-{h} \\
\Phi'_1 : (\Delta^3 V^* \times X^3) \, \rcap_{(E^*)^3} \, (F^\bot)^3 \ar[u]^-{\Phi_2'} \ar[r]^-{\Phi_1'} & (\overline{\Delta} V^* \times X^3) \, \rcap_{E^* \times X \times E^*} (F^\bot \times X \times F^\bot). \ar[u]_-{\Omega}
}
}
\]
Both squares in this diagram are cartesian in the (derived) sense of \cite[\S3.7]{BR}, and we have $\Psi_1= h \circ \Omega$. We deduce isomorphisms of functors
\[
L(\Psi_1)^* \circ R(\Psi_2)_* \circ Rg_* \ \cong \ L\Omega^* \circ Lh^* \circ R(\Psi_2 \circ g)_* \ \cong \ L\Omega^* \circ R\Theta_* \circ L(f')^* \ \cong \ R(\Phi'_1)_* \circ L(\Phi'_2)^*\circ L(f')^*.
\]
Here the first isomorphism follows from standard properties of (left and right) derived functors of compositions, and the second and third isomorphisms follow from an equivariant version of \cite[Proposition 3.7.1]{BR} (which can be deduced from the non-equivariant version by the techniques of Section \ref{sec:G-equiv}). This finishes the proof of \eqref{eqn:isom-convolution-LKD}, and so the proof of the proposition.
\end{proof}

\begin{remark}
\label{rk:squareroot}
\begin{enumerate}
\item
As in Remark \ref{rk:qis}, $\Phi_2'$ and $\Omega$ are quasi-isomorphisms of dg-schemes, and are in fact induced by isomorphisms of dg-algebras. Hence the second isomorphism in the last equation of the proof can be checked in an elementary way, without referring to \cite{BR}. Writing all the morphisms explicitly in terms of dg-algebras, one can also check the second isomorphism in this equation elementarily: it is just a statement about $K$-flat extensions and restrictions of scalars.
\item
Assume that the line bundle $\omega_X$ has a $G$-equivariant square root, \emph{i.e.}~there exists a $G$-equivariant line bundle $\omega_X^{1/2}$ on $X$ such that $(\omega_X^{1/2})^{\otimes 2} \cong \omega_X$. Then one can define a modified equivalence $\frakK'$ using the dualizing complex $\omega_X^{-1/2} \boxtimes \omega_X^{-1/2}[d_X]$, without affecting Proposition~\ref{prop:convolution} (nor Proposition~\ref{prop:diagonal} below). This provides a more symmetric definition of $\frakK$ in this case.
\end{enumerate}
\end{remark}


\subsection{Image of the unit}

As in \S\ref{ss:compatibilityconvolution} we consider the equivalence $\frakK$. Let us denote by $q : E^2 \to X^2$ the projection. Consider the structure sheaf of the diagonal copy of $F$ in $E^2$, denoted $\calO_{\Delta F}$. Then $q_* \calO_{\Delta F}$ is an object of the category
\[
\calD^{\mathrm{c}}_{G \times \Gm} \bigl( (\Delta V \times X \times X) \, \rcap_{E \times E} \, (F \times F) \bigr),
\]
where the structure of $\mathrm{S}_{\calO_{X^2}}(\calF^{\vee} \boxplus \calF^{\vee}) \otimes_{\C} (\bigwedge V^*)$-dg-module is given by the composition of $\mathrm{S}_{\calO_{X^2}}(\calF^{\vee} \boxplus \calF^{\vee}) \otimes_{\C} (\bigwedge V^*) \to q_* \calO_{F \times_V F}$ (projection to the $0$-cohomology) and $q_* \calO_{F \times_V F} \to q_* \calO_{\Delta F}$ (restriction). For simplicity, in the rest of this subsection we write $\calO_{\Delta F}$ for $q_* \calO_{\Delta F}$. Similarly we have an object $\calO_{\Delta F^{\bot}}$ in $\calD^{\mathrm{c}}_{G \times \Gm} \bigl( (\Delta V^* \times X \times X) \, \rcap_{E^* \times E^*} \, (F^{\bot} \times F^{\bot}) \bigr)$.

The idea of the proof of the following proposition is, using isomorphisms of functors proved in Propositions~\ref{prop:inclusion} and~\ref{prop:basechange}, to reduce the claim to an explicit and easy computation.

\begin{prop}
\label{prop:diagonal}
We have $\frakK(\calO_{\Delta F}) \cong \calO_{\Delta F^{\bot}}$.
\end{prop}

\begin{proof}
Consider the morphism $\delta : X \hookrightarrow X \times X$ (inclusion of the diagonal). We denote by
\[
\kappa_{\delta} : \calD^{\mathrm{c}}_{G \times \Gm} \bigl( (\Delta V \times X) \, \rcap_{E \times_X E} \, (F \times_X F) \bigr) \ \xrightarrow{\sim} \ \calD^{\mathrm{c}}_{G \times \Gm} \bigl( (\overline{\Delta} V^* \times X) \, \rcap_{E^* \times_X E^*} \, (F^{\bot} \times_X F^{\bot}) \bigr)^\op
\]
the linear Koszul duality equivalence associated with the dualizing complex $\delta^!(\omega_X \boxtimes \calO_X [d_X]) \cong \calO_X$ on $X$. By Proposition~\ref{prop:basechange}, there is an isomorphism of functors
\begin{equation}
\label{eq:isomDelta}
\kappa \circ R\hat{\delta}_* \cong R\tilde{\delta}_* \circ \kappa_{\delta},
\end{equation}
where the functors $R\hat{\delta}_*$ and $R\tilde{\delta}_*$ are defined as in \S\ref{ss:base-change-lkd}. 

Consider the object $\mathrm{S}_{\calO_X}(\calF^{\vee})$ of the category
\[
\calD^{\mathrm{c}}_{G \times \Gm} \bigl( (\Delta V \times X) \, \rcap_{E \times_X E} \, (F \times_X F) \bigr) \ \cong \ \calD^\fg \bigl( \Sym(V^* \otimes_{\C} \calO_X \to \calF^{\vee} \oplus \calF^{\vee})\Mod^G \bigr),
\]
where the dg-module structure corresponds to the diagonal inclusion $\calF \hookrightarrow \calF \oplus \calF$. Then by definition $\calO_{\Delta F} \cong R\hat{\delta}_* \bigl( \mathrm{S}_{\calO_X}(\calF^{\vee}) \bigr)$. Hence, using isomorphism~\eqref{eq:isomDelta}, we obtain
\begin{equation}
\label{eq:K(O_D)}
\frakK(\calO_{\Delta F}) \ = \ \Xi \circ \kappa(\calO_{\Delta F}) \ \cong \ \Xi \circ R \tilde{\delta}_* \circ \kappa_{\delta} (\mathrm{S}_{\calO_X}(\calF^{\vee})),
\end{equation}
where $\Xi$ is defined as in \S\ref{ss:defconvolution}.

Now consider the diagonal embedding $F^{\mathrm{diag}} \hookrightarrow F \times_X F$ as in \S\ref{ss:compatibilityconvolution}. This inclusion makes $F^{\mathrm{diag}}$ a subbundle of $E \times_X E$. Taking the derived intersection with $\Delta V \times X$ inside $E \times_X E$, we are in the setting of \S\ref{ss:subbundles}. We consider the morphisms of dg-schemes
\begin{align*}
\underline{f}: (\Delta V \times X) \, \rcap_{E \times_X E} \, F^{\mathrm{diag}} \ & \to \ (\Delta V \times X) \, \rcap_{E \times_X E} \, (F \times_X F), \\
\underline{g}: (\overline{\Delta} V^* \times X) \, \rcap_{E^* \times_X E^*} \, (F^{\bot} \times_X F^{\bot}) \ & \to \ (\overline{\Delta} V^* \times X) \, \rcap_{E^* \times_X E^*} \, (F^{\mathrm{diag}})^{\bot},
\end{align*}
and the diagram:
\[
\xymatrix{
\calD^{\mathrm{c}}_{G \times \Gm}((\Delta V \times X) \, \rcap_{E \times_X E} \, F^{\mathrm{diag}}) \ar[r]^-{\kappa_F}_-{\sim} \ar[d]^-{R(\underline{f})_*} & \calD^{\mathrm{c}}_{G \times \Gm}((\overline{\Delta} V^* \times X) \, \rcap_{E^* \times_X E^*} \, (F^{\mathrm{diag}})^{\bot})^\op \ar[d]^-{L(\underline{g})^*} \\
\calD^{\mathrm{c}}_{G \times \Gm}((\Delta V \times X) \, \rcap_{E \times_X E} \, (F \times_X F)) \ar[r]^-{\kappa_{\delta}}_-{\sim} & \calD^{\mathrm{c}}_{G \times \Gm}((\overline{\Delta} V^* \times X) \, \rcap_{E^* \times_X E^*} \, (F^{\bot} \times_X F^{\bot}))^\op
}
\]
where $\kappa_F$ is the linear Koszul duality equivalence associated with the dualizing complex $\calO_X$ on $X$.
The structure (dg-)sheaf of $(\Delta V \times X) \, \rcap_{E \times_X E} \, F^{\mathrm{diag}}$ is $(\bigwedge V^*) \otimes_{\C} \mathrm{S}_{\calO_X}(\calF^{\vee})$, with trivial differential (because $F^{\mathrm{diag}} \subset \Delta V \times X$). In particular, $\mathrm{S}_{\calO_X}(\calF^{\vee})$ is also an object of the top left category in the diagram, which we denote by $\calO_F$. Then, by definition, $R(\underline{f})_* \calO_F$ is the object $\mathrm{S}_{\calO_X}(\calF^{\vee})$ appearing in~\eqref{eq:K(O_D)}. By Proposition~\ref{prop:inclusion}, there is an isomorphism of functors
\begin{equation}
\label{eqn:isom-again}
\kappa_{\delta} \circ R(\underline{f})_* \ \cong \ L(\underline{g})^* \circ \kappa_F.
\end{equation}
In particular we have $\kappa_{\delta} (\mathrm{S}_{\calO_X}(\calF^{\vee})) \cong L(\underline{g})^* \circ \kappa_F (\calO_F)$.

The object $\kappa_F(\calO_F)$ is the $(\bigwedge \calF) \otimes_{\C} \mathrm{S}(V)$-dg-module $\calO_X \otimes_\C \mathrm{S}(V)$, and the morphism $\underline{g}$ is associated with the morphism of dg-algebras
\[
(\bigwedge \calF) \otimes_{\C} \mathrm{S}(V) \to \bigl( \bigwedge(\calF \oplus \calF) \bigr) \otimes_{\C} \mathrm{S}(V)
\]
on $X$
induced by the diagonal embedding $\calF \hookrightarrow \calF \oplus \calF$, which makes the right hand dg-algebra $K$-flat over the left hand dg-algebra. (Here $V$ is identified with the quotient $(V \oplus V) / \Delta V$ via the morphism $(a,b) \mapsto a-b$.) It follows that $L(\underline{g})^*(\calO_X \otimes_\C \mathrm{S}(V))$ is isomorphic to the (non-derived) tensor product
\[
\bigl( \bigwedge(\calF \oplus \calF)  \otimes_{\C} \mathrm{S}(V) \bigr) \otimes_{(\bigwedge \calF) \otimes_{\C} \mathrm{S}(V)} \bigl( \calO_X \otimes_\C \mathrm{S}(V) \bigr) \cong \bigwedge(\calF \oplus \calF / \calF)  \otimes_{\C} \mathrm{S}(V).
\]
One can easily check that the right-hand side is simply the Koszul complex associated with the inclusion $\calF \hookrightarrow V \otimes \calO_X$, and hence that it is quasi-isomorphic to 
\[
\mathrm{S}_{\calO_X} \bigl( (V \otimes \calO_X) / \calF \bigr) \cong \mathrm{S}_{\calO_X} \bigl( (\calF^\bot)^\vee \bigr).
\]
Using isomorphisms~\eqref{eq:K(O_D)} and~\eqref{eqn:isom-again}, we deduce the isomorphism of the proposition.
\end{proof}


From Proposition~\ref{prop:diagonal} one deduces the following result.

\begin{cor} \label{cor:imagelinebundle}
Let $\calL$ be a $G$-equivariant line bundle on $X$. Then
$\calO_{\Delta F} \otimes_{\calO_X} \calL$ is naturally an object of
$\calD^{\mathrm{c}}_{G \times \Gm} \bigl( (\Delta V \times X \times X)
\, \rcap_{E \times E} \, (F \times F) \bigr)$. We have $\frakK(\calO_{\Delta
  F} \otimes_{\calO_X} \calL) \cong \calO_{\Delta F^{\bot}}
\otimes_{\calO_X} \calL^{\vee}$.
\end{cor}

\section{Linear Koszul Duality and Iwahori--Matsumoto Involution}
\label{sec:IMinvolution}

\subsection{Contractibility}

Let $X$ be a complex algebraic variety endowed with an action of an algebraic group $G$, and let $\calA$ be a $G \times \Gm$-equivariant sheaf of quasi-coherent dg-algebras on $X$,
bounded and concentrated in non-positive degrees (for the cohomological grading). Assume that $\calH^0(\calA)$ is locally finitely generated as an $\calO_X$-algebra, and that $\calH^\bullet(\calA)$ 
is locally finitely generated as an $\calH^0(\calA)$-module. 
Consider the triangulated category $\calD^\fg(\calA\Mod^{G})$, and let
$\Kth^{G \times \Gm}(\calA)$ be
its Grothen\-dieck group. Let also $\Kth^{G \times \Gm}(\calH^0(\calA))$ be the
Grothendieck group of the abelian category of $G \times \Gm$-equivariant quasi-coherent, locally
finitely generated $\calH^0(\calA)$-modules.

\begin{lem}
\label{lem:Grothendieckgroups}
The natural morphism
\[
\left\{
\begin{array}{ccc}
\Kth^{G \times \Gm}(\calA) & \to & \Kth^{G \times \Gm}(\calH^0(\calA)) \\
{[} \calM] & \mapsto & \sum_{i \in \mathbb{Z}} (-1)^i \cdot [ \calH^i(\calM) ]
\end{array}
\right.
\]
is an isomorphism of abelian groups.
\end{lem}

\begin{proof} 
Every object of $\calD^{\fg}(\calA\Mod^G)$ is
  isomorphic to an object which is a bounded
  $\calA$-dg-module for the cohomological grading. (This follows from the fact that $\calA$ is
  bounded and concentrated in non-positive degrees, using truncation functors, as defined \emph{e.g.}~in
  \cite[\S 2.1]{MR}.) So let $\calM$ be an object of $\calD^{\fg}(\calA\Mod^G)$ such that $\calM^j=0$ for $j \notin
  \llbracket a,b \rrbracket$ for some integers $a<b$. Let $n=b-a$. Consider the
  following filtration of $\calM$ as an $\calA$-dg-module:
\[
\{0\}=\calM_{-1} \subset \calM_0 \subset \calM_1 \subset \cdots \subset \calM_n=\calM,
\] where for $j \in \llbracket 0,n \rrbracket$ we put
\[
\calM_j:=(\cdots \ 0 \to \calM^a \to \cdots \to \calM^{a+j-1} \xrightarrow{d^{a+j-1}} \mathrm{Ker}(d^{a+j}) \to 0 \ \cdots).
\]
Then, in $\Kth^{G \times \Gm}(\calA)$ we have
\[
[ \calM ] = \sum_{j=0}^{n} \, [ \calM_j / \calM_{j-1} ] = \sum_{i \in \mathbb{Z}} \, (-1)^i \cdot [ \calH^i(\calM) ],
\]
where $\calH^i(\calM)$ is considered as an $\calA$-dg-module concentrated in degree $0$. It follows that the natural morphism $\Kth^{G \times \Gm}(\calH^0(\calA)) \to \Kth^{G \times \Gm} (\calA)$, which sends an $\calH^0(\calA)$-module to itself, viewed as an $\calA$-dg-module concentrated in degree $0$, is an inverse to the morphism of the lemma.
\end{proof}

\subsection{Reminder on affine Hecke algebras}
\label{ss:reminderHecke}

Now we assume that $G$ is a connected, simply-connected, complex semi-simple algebraic group. Let $T \subset B \subset G$ be a torus and a Borel subgroup of $G$. Let also $\frakt \subset \frakb \subset \frakg$ be their Lie algebras. Let $U$ be the unipotent radical of $B$, and let $\frakn$ be its Lie algebra. Let $\calB:=G/B$ be the flag variety of $G$. Consider the Springer variety $\wcalN$ and the Grothendieck resolution $\wfrakg$ defined as follows:
\[
\wcalN := \{(X,gB) \in \frakg^* \times \calB \mid X_{|g \cdot \frakb}=0 \}, \qquad \wfrakg := \{(X,gB) \in \frakg^* \times \calB \mid X_{|g \cdot \frakn}=0\}.
\]
(The variety $\wcalN$ is naturally isomorphic to the cotangent bundle of $\calB$.) The varieties $\wcalN$ and $\wfrakg$ are subbundles of the trivial vector bundle $\frakg^* \times \calB$ over $\calB$. In particular, there are natural maps $\wcalN \to \frakg^*$ and $\wfrakg \to \frakg^*$. Let us consider the varieties
\[
Z:=\wcalN \times_{\frakg^*} \wcalN, \qquad \mathcal{Z}:=\wfrakg \times_{\frakg^*} \wfrakg.
\]
There is a natural action of $G \times \Gm$ on $\frakg^* \times \calB$, where $(g,t)$ acts via:
\[
(g,t) \cdot (X,hB) \ := \ (t^{-2} (g \cdot X), ghB).
\]
The subbundles $\wcalN$ and $\wfrakg$ are $G \times \Gm$-stable. For any variety $X \to \calB$ over $\calB$ and for $x \in \bbX$, we denote by $\calO_X(x)$ the pullback to $X$ of the line bundle on $\calB$ associated with $x$. We use a similar notation for varieties over $\calB \times \calB$ (and pairs of weights).

Let $R$ be the root system of $G$, $R^+$ the positive roots (chosen as the weights of $\frakg / \frakb$), $S \subset R^+$ the associated set of simple roots, $\bbX$ the weights of $R$ (which naturally identify with the group of characters of $T$). Let also $W$ be the Weyl group of $R$ (or of $(G,T)$). For $\alpha \in S$ we denote by $s_{\alpha} \in W$ the corresponding simple reflection. For $\alpha,\beta \in S$, we let $n_{\alpha,\beta}$ be the order of $s_{\alpha} s_{\beta}$ in $W$. Then the (extended) affine Hecke algebra $\calH_{\aff}$ associated with these data is the $\mathbb{Z}[v,v^{-1}]$-algebra generated by elements $\{T_{\alpha}, \, \alpha \in S\} \cup \{\theta_x, \, x \in \bbX\}$, with defining relations
\begin{align*}
\rmi & \quad T_{\alpha} T_{\beta} \cdots = T_{\beta} T_{\alpha} \cdots \quad (n_{\alpha,\beta} \ \text{elements on each side}) \\
\rmii & \quad \theta_0=1 \\
\rmiii & \quad \theta_x \theta_y = \theta_{x+y} \\
\rmiv & \quad T_{\alpha} \theta_x = \theta_x T_{\alpha} \quad \text{if} \ s_{\alpha}(x)=x \\
\rmv & \quad \theta_x=T_{\alpha} \theta_{x - \alpha} T_{\alpha} \quad \text{if} \ s_{\alpha}(x)=x-\alpha \\
\rmvi & \quad (T_{\alpha} + v^{-1})(T_{\alpha} - v)=0
\end{align*}
for $\alpha,\beta \in S$ and $x,y \in \bbX$ (see \emph{e.g.}~\cite{CG, L}).

We will be interested in the \emph{Iwahori--Matsumoto involution} $\mathrm{IM}$ of $\calH_{\aff}$. This is the involution of $\mathbb{Z}[v,v^{-1}]$-algebra of $\calH_{\aff}$ defined on the generators by
\[
\left\{
\begin{array}{ccc} \mathrm{IM}(T_{\alpha}) & = & - T_{\alpha}^{-1}, \\
\mathrm{IM}(\theta_x) & = & \theta_{-x}. \end{array}
\right.
\]
For $\alpha \in S$ we also consider $t_{\alpha}:=v \cdot T_{\alpha}$. Then we have $\mathrm{IM}(t_{\alpha})=-q(t_{\alpha})^{-1}$, with $q=v^2$.

Let $\alpha \in S$. Let $P_{\alpha} \supset B$ be the minimal standard parabolic subgroup associated with $\alpha$, let $\frakp_{\alpha}$ be its Lie algebra, and let $\calP_{\alpha}:=G/P_{\alpha}$ be the associated partial flag variety. We define the following $G \times \Gm$-subvariety of $Z$:
\[
Y_{\alpha}:=\{(X,g_1 B, g_2 B) \in \frakg^* \times (\calB \times_{\calP_{\alpha}} \calB) \mid X_{|g \cdot \frakp_{\alpha}}=0 \}.
\]
We also let
\[
\wfrakg_{\alpha} := \{ (X,gP_{\alpha}) \in \frakg^* \times \calP_{\alpha} \mid X_{|g \cdot \frakp_{\alpha}^{\mathrm{n}}}=0 \},
\]
where $\frakp_{\alpha}^{\mathrm{n}}$ is the nilpotent radical of $\frakp_{\alpha}$. There is a natural morphism $\wfrakg \to \wfrakg_{\alpha}$.

By a result of Kazhdan--Lusztig (\cite{KL}; see also \cite{CG,L}) there is a natural isomorphism of $\mathbb{Z}[v,v^{-1}]$-algebras
\begin{equation}
\label{eq:isomHaff}
\calH_{\aff} \ \xrightarrow{\sim} \Kth^{G \times \Gm}(Z),
\end{equation}
where the equivariant $\Kth$-theory $\Kth^{G \times \Gm}(Z)$ is endowed with the convolution product associated with the embedding $Z \subset \wcalN \times \wcalN$.
Isomorphism~\eqref{eq:isomHaff} can be defined by
\[
\left\{
\begin{array}{ccc} T_{\alpha} & \mapsto & -
    v^{-1} [\calO_{Y_{\alpha}}(-\rho, \rho-\alpha)] - v^{-1} [\Delta_* \calO_{\wcalN}] \\[2pt]
    \theta_x & \mapsto & [\Delta_* \calO_{\wcalN}(x)]
  \end{array}
\right.
\]
for $\alpha \in S$ and $x \in \bbX$. Here, $\Delta: \wcalN \hookrightarrow Z$ is the diagonal embedding, and for $\calF$ in $\Coh^{G \times \Gm}(Z)$ we denote by $[\calF]$ its class in $\Kth$-theory. The action of $v$ is induced by the functor $\langle 1 \rangle : \Coh^{G \times \Gm}(Z) \to \Coh^{G \times \Gm}(Z)$ of tensoring with the one-dimensional tautological $\Gm$-module.

Consider the embedding of smooth varieties $i : \wcalN \times \wfrakg \hookrightarrow \wfrakg \times \wfrakg$. Associated with this morphism, there is a morphism of ``restriction with supports'' in $\Kth$-theory
\[
\eta : \Kth^{G \times \Gm}(\mathcal{Z}) \to \Kth^{G \times \Gm}(Z)
\]
(see \cite[p.~246]{CG}). As above for $\Kth^{G \times \Gm}(Z)$, convolution endows $\Kth^{G \times \Gm}(\mathcal{Z})$ with the structure of a $\mathbb{Z}[v,v^{-1}]$-algebra. (Here we use the embedding $\mathcal{Z} \subset \wfrakg \times \wfrakg$ to define the product.) The following result is well known. As we could not find a reference, we include a proof.

\begin{lem} \label{lem:isomKth}
The morphism $\eta$ is an isomorphism of $\mathbb{Z}[v,v^{-1}]$-algebras.
\end{lem}

\begin{proof}
Let us denote by $j : \wcalN \times \wcalN \hookrightarrow \wcalN \times \wfrakg$ and $k : \wcalN \hookrightarrow \wfrakg$ the embeddings. Let also $\Gamma_k$ be the graph of $k$. Then $\eta$ is the composition of the morphism in $\Kth$-theory induced by the functor
\[
Lj^* : \calD^b \Coh_{\mathcal{Z}}^{G \times \Gm}(\wfrakg \times \wfrakg) \to \calD^b \Coh_Z^{G \times \Gm}(\wcalN \times \wfrakg)
\]
and by the inverse of the isomorphism induced by
\[
i_* : \calD^b \Coh_Z^{G \times \Gm}(\wcalN \times \wcalN) \to \calD^b \Coh_Z^{G \times \Gm}(\wcalN \times \wfrakg).
\]
By \cite[Lemma 1.2.3]{R1}, $i_*$ is the product on the left (for convolution) by the kernel $\calO_{\Gamma_k} \in \calD^b \Coh^{G \times \Gm}(\wcalN \times \wfrakg)$. By similar arguments, $Lj^*$ is the product on the right by the kernel $\calO_{\Gamma_k}$. It follows from these observations that $\eta$ is a morphism of $\mathbb{Z}[v,v^{-1}]$-algebras.

Then we observe that $Z$ and $\mathcal{Z}$ have compatible cellular fibrations (in the sense of \cite[\S 5.5]{CG}), defined using the partition of $\calB \times \calB$ into $G$-orbits. The stratas in $Z$ are the transverse intersections of those of $\mathcal{Z}$ with $\wcalN \times \wfrakg \subset \wfrakg \times \wfrakg$. It follows, using the arguments of \cite[\S 6.2]{CG}, that $\eta$ is an isomorphism of $\mathbb{Z}[v,v^{-1}]$-modules, completing the proof.\end{proof}

It follows in particular from this lemma that there is also an isomorphism
\begin{equation}
\label{eq:Haff1}
\calH_{\aff} \ \xrightarrow{\sim} \ \Kth^{G \times \Gm}(\calZ),
\end{equation}
which satisfies
\[
\left\{
\begin{array}{ccc} T_{\alpha} & \mapsto & -
    v^{-1} [\calO_{\wfrakg \times_{\wfrakg_{\alpha}} \wfrakg}] + v [\Delta_* \calO_{\wfrakg}] \\[2pt]
    \theta_x & \mapsto & [\Delta_* \calO_{\wfrakg}(x)]
  \end{array}
\right.
\]
(see \emph{e.g.}~\cite{BR, R1} for details). This is not exactly the isomorphism we are going to use. Instead, observe that the tensor product with the line bundle $\calO_{\wfrakg \times \wfrakg}(-\rho,\rho)$ induces an algebra automorphism of $\Kth^{G \times \Gm}(\calZ)$. Hence there exists an isomorphism
\begin{equation}
\label{eq:Haff2}
\calH_{\aff} \ \xrightarrow{\sim} \ \Kth^{G \times \Gm}(\calZ),
\end{equation}
which satisfies
\[
\left\{
\begin{array}{ccc} T_{\alpha} & \mapsto & -
    v^{-1} [\calO_{\wfrakg \times_{\wfrakg_{\alpha}} \wfrakg}(-\rho,\rho)] + v [\Delta_* \calO_{\wfrakg}] \\[2pt]
    \theta_x & \mapsto & [\Delta_* \calO_{\wfrakg}(x)]
  \end{array}
\right. .
\]
We will rather use the latter isomorphism.

Finally, we define $N:=\#(R^+)=\dim(\calB)$.

\subsection{Geometric realization of the Iwahori--Matsumoto involution}
\label{ss:LKDIM}

From now on we consider a very special case of linear Koszul
duality, namely the situation of Section~\ref{sec:convolution} with $X=\calB$,
$V=\frakg^*$ and $F=\wcalN$. We have $\calF=\calT_\calB^\vee$, where $\calT_\calB$ is the tangent sheaf on $\calB$. We identify $V^*=\frakg$ with $\frakg^*$
using the Killing form. Then $F^{\bot}$ identifies with $\wfrakg$. We
obtain an equivalence
\[
\frakK : \calD^{\mathrm{c}}_{G \times \Gm}
\bigl( (\Delta \frakg^* \times \calB \times \calB) \, \rcap_{(\frakg^*
  \times \calB)^2} \, (\wcalN \times \wcalN) \bigr) \ \xrightarrow{\sim} \
\calD^{\mathrm{c}}_{G \times \Gm} \bigl( (\Delta \frakg^* \times \calB
\times \calB) \, \rcap_{(\frakg^* \times \calB)^2} \, (\wfrakg \times
\wfrakg) \bigr)^\op.
\]
Here the actions of $\Gm$ on
$\frakg^*$ are not the same on the two sides: they are ``inverse'', \emph{i.e.}~each one is
the composition of the other
one with $t \mapsto t^{-1}$. We denote by $\frakK_{\mathrm{IM}}$ the composition
of $\frakK$ with the auto-equivalence of $\calD^{\mathrm{c}}_{G \times \Gm}
\bigl( (\Delta \frakg^* \times \calB \times \calB) \, \rcap_{(\frakg^*
  \times \calB)^2} \, (\wfrakg \times \wfrakg) \bigr)$ which inverts
the $\Gm$-action. (In the realization as $\Gm$-equivariant dg-modules on $\calB \times \calB$, this
amounts to inverting the internal grading.)

By Lemma~\ref{lem:Grothendieckgroups} (together with the fact that the direct image morphism $\Kth^{G \times \Gm}(Z) \to \Kth^{G \times \Gm}(\wcalN \times_{\frakg^*} \wcalN)$ is an isomorphism), the Grothendieck group of the
category $\calD^{\mathrm{c}}_{G \times \Gm} \bigl(
(\Delta \frakg^* \times \calB \times \calB) \, \rcap_{(\frakg^* \times
  \calB)^2} \, (\wcalN \times \wcalN) \bigr)$ is naturally isomorphic
to $\Kth^{G \times \Gm}(Z)$, which itself identifies with the affine Hecke algebra
$\calH_{\aff}$ (see~\eqref{eq:isomHaff}). Similarly, the Grothendieck group
of the category $\calD^{\mathrm{c}}_{G \times \Gm} \bigl( (\Delta
\frakg^* \times \calB \times \calB) \, \rcap_{(\frakg^* \times
  \calB)^2} \, (\wfrakg \times \wfrakg) \bigr)$ is isomorphic to $\Kth^{G
  \times \Gm}(\calZ)$; hence it is also isomorphic to $\calH_{\aff}$ (using isomorphism~\eqref{eq:Haff2}). One can easily check that the convolution product on derived categories of dg-sheaves defined in Section~\ref{sec:convolution} induces the convolution product in $\Kth$-theory considered in \S\ref{ss:reminderHecke}, so that these isomorphisms are algebra isomorphisms. Let us consider the automorphism $[\frakK_{\mathrm{IM}}] : \calH_{\aff} \to \calH_{\aff}$ induced by $\frakK_{\mathrm{IM}}$.
  
\begin{remark}
\label{rk:calZ-derived}
\begin{enumerate}
\item
A simple dimension
  counting as in the proof of Proposition~\ref{prop:imageYalpha} below shows that the derived intersection $(\Delta \frakg^*
  \times \calB \times \calB) \, \rcap_{(\frakg^* \times \calB)^2} \,
  (\wfrakg \times \wfrakg)$ is quasi-isomorphic, as a dg-scheme, to
  $\calZ$. Hence we do not really need Lemma~\ref{lem:Grothendieckgroups} to identify the Grothendieck group of $\calD^{\mathrm{c}}_{G \times \Gm} \bigl( (\Delta
\frakg^* \times \calB \times \calB) \, \rcap_{(\frakg^* \times
  \calB)^2} \, (\wfrakg \times \wfrakg) \bigr)$ with $\Kth^{G \times \Gm}(\calZ)$. In fact, it follows from \cite[Proposition 1.3.2]{MR} that there exists a natural equivalence of triangulated categories
\[
\calD^{\mathrm{c}}_{G \times \Gm} \bigl( (\Delta
\frakg^* \times \calB \times \calB) \, \rcap_{(\frakg^* \times
  \calB)^2} \, (\wfrakg \times \wfrakg) \bigr) \ \cong \
  \calD^b \Coh^{G \times \Gm}(\calZ).
\]
The similar statement for $\wcalN$ does \emph{not} hold, however.
\item
We have $\omega_{\calB}=\calO_{\calB}(-2\rho)$; in particular, this sheaf has a $G$-equivariant square root. Using Remark~\ref{rk:squareroot}, we could have used a more symmetric equivalence $\frakK'$ instead of $\frakK$. Theorem~\ref{thm:mainthm} below remains true if we replace $\frakK$ by $\frakK'$ and isomorphism~\eqref{eq:Haff2} by isomorphism~\eqref{eq:Haff1}.
\end{enumerate}
\end{remark}

In the presentation of $\calH_{\aff}$ using the generators $t_{\alpha}$ and $\theta_x$, the scalars appearing in the relations are polynomials in $q=v^2$. Hence we can define the involution $\iota$ of
$\calH_{\aff}$ (as an algebra) that fixes the $t_{\alpha}$'s and $\theta_x$'s, and sends
$v$ to $-v$. Note that we have $\iota \circ \mathrm{IM} = \mathrm{IM} \circ \iota$.

The main result of this section is the following.

\begin{thm}
\label{thm:mainthm}
The automorphism $[\frakK_{\mathrm{IM}}]$ of $\calH_{\aff}$ is the composition of the Iwahori--Matsumoto involution $\mathrm{IM}$ and the
involution $\iota$:
\[
[\frakK_{\mathrm{IM}}] \ = \ \iota \circ \mathrm{IM}
\]
as automorphisms of $\calH_{\aff}$.
\end{thm}

\begin{remark}
\label{rk:conjecture-bezru}
Bezrukavnikov has conjectured in \cite[\S11.4]{bezru} a 
strengthening of Theorem \ref{thm:mainthm}.
The main result of \cite{bezru}
is the identification of two categorifications
of affine Hecke algebras,
by coherent sheaves on the Steinberg variety and by
perverse sheaves on
the affine flag variety of the Langlands dual $\check G$ of $G$.
In \cite{BY}, Bezrukavnikov and Yun
categorify the Iwahori--Matsumoto involution
by the affine version of the Be{\u\i}linson--Ginzburg--Soergel
Koszul duality in category $\mathcal{O}$.
The conjecture identifies the two categorifications of
the involution, namely
a version of our linear Koszul duality of
coherent sheaves on the Steinberg variety
and the
affine 
Koszul duality for category $\mathcal{O}$.


To explain this conjecture more precisely we first replace the equivalence $\frakK$ by its variant $\frakK'$ as in Remark \ref{rk:calZ-derived}; we denote by $\frakK'_{\mathrm{IM}}$ the resulting equivalence. Using similar arguments as above one can construct a linear Koszul duality equivalence $\frakK''_{\mathrm{IM}}$ such that the following diagram commutes:
\[
\xymatrix@C=2cm{
\calD^{\mathrm{c}}_{G \times \Gm}
\bigl( (\Delta \frakg^* \times \calB \times \calB) \, \rcap_{(\frakg^*
  \times \calB)^2} \, (\wcalN \times \wfrakg) \bigr) \ar[r]^-{\frakK''_{\mathrm{IM}}} &
\calD^{\mathrm{c}}_{G \times \Gm} \bigl( (\Delta \frakg^* \times \calB
\times \calB) \, \rcap_{(\frakg^* \times \calB)^2} \, (\wfrakg \times
\wcalN) \bigr)^\op \\
\calD^{\mathrm{c}}_{G \times \Gm}
\bigl( (\Delta \frakg^* \times \calB \times \calB) \, \rcap_{(\frakg^*
  \times \calB)^2} \, (\wcalN \times \wcalN) \bigr) \ar[r]^{\frakK'_{\mathrm{IM}}} \ar[u]^-{(\cdot)_*} &
\calD^{\mathrm{c}}_{G \times \Gm} \bigl( (\Delta \frakg^* \times \calB
\times \calB) \, \rcap_{(\frakg^* \times \calB)^2} \, (\wfrakg \times
\wfrakg) \bigr)^\op. \ar[u]_-{(\cdot)^*}
}
\]
Here the left vertical arrow is a direct image functor and the right vertical arrow is an inverse image functor, both defined as in \S\ref{ss:subbundles}. The same constructions as in Section \ref{sec:convolution} can be used to endow the top left (resp.~right) category with a right action of the bottom left (resp.~right) category; with these definitions one has $\frakK''_{\mathrm{IM}}(\calF \star \calG) \cong \frakK''_{\mathrm{IM}}(\calF) \star \frakK'_{\mathrm{IM}}(\calG)$. Both vertical arrows induce isomorphisms in $\Kth$-theory. One can deduce from Theorem \ref{thm:mainthm} that the upper line in the diagram ``categorifies'' the composition $\iota \circ \mathrm{IM}$ as an automorphism of right modules $\calH_{\aff} \xrightarrow{\sim} (\calH_{\aff})^{\iota \circ \mathrm{IM}}$; here the superscript indicates a twist of the action.

Bezrukavnikov conjectures that, under a ``mixed version'' of equivalence $(3)$ in \cite{bezru}, the equivalence $\frakK''_{\mathrm{IM}}$ corresponds to the composition of the Koszul duality constructed in \cite[\S5.3]{BY} and Verdier duality.

Note that, using the same arguments as in Remark \ref{rk:calZ-derived}, one can check that the domain, resp.~codomain, of the equivalence $\frakK''_{\mathrm{IM}}$ is equivalent to $\calD^b \Coh^{G \times \Gm}(\wcalN \times_{\frakg^*} \wfrakg)$, resp.~to $\calD^b \Coh^{G \times \Gm}(\wfrakg \times_{\frakg^*} \wcalN)$. Hence this conjecture can be stated in terms not involving any derived intersection. 
However, this non-derived formulation is not symmetric in two factors
of fibered products so it does not provide 
a direct approach to convolution.
\end{remark}

We will prove Theorem \ref{thm:mainthm} in \S\ref{ss:endproof}. Before that, we need one more preliminary result.

Let $\alpha$ be a simple root. The coherent sheaf
$\calO_{Y_{\alpha}}(\rho - \alpha, -\rho)$ on $Z$ has a natural structure of $G \times \Gm$-equivariant dg-module over $\mathrm{S}( \calT_{\calB} \boxplus
\calT_{\calB} ) \otimes_\C (\bigwedge \frakg)$; hence it defines an object in the category $\calD^{\mathrm{c}}_{G
  \times \Gm} \bigl( (\Delta \frakg^* \times \calB \times \calB) \,
\rcap_{(\frakg^* \times \calB)^2} \, (\wcalN \times \wcalN)
\bigr)$. Similarly, $\calO_{\wfrakg \times_{\wfrakg_{\alpha}}
  \wfrakg}$ is naturally an object of $\calD^{\mathrm{c}}_{G \times
  \Gm} \bigl( (\Delta \frakg^* \times \calB \times \calB) \,
\rcap_{(\frakg^* \times \calB)^2} \, (\wfrakg \times \wfrakg) \bigr)$. The proof of the next proposition is very similar to that of Proposition~\ref{prop:diagonal}.

\begin{prop}
\label{prop:imageYalpha}
We have $\frakK(\calO_{Y_{\alpha}}(\rho - \alpha, -\rho)) \ \cong \ \calO_{\wfrakg \times_{\wfrakg_{\alpha}} \wfrakg}(-\rho,\rho)[1]$.
\end{prop}

\begin{proof}
First we observe that $\calO_{Y_{\alpha}}(\rho - \alpha, -\rho) \cong \calO_{Y_{\alpha}}(-\rho, \rho-\alpha)$ (see \cite[Lemma 1.5.1]{R1}). Hence to prove the proposition it is sufficient to prove that $\frakK(\calO_{Y_{\alpha}}) \cong \calO_{\wfrakg \times_{\wfrakg_{\alpha}} \wfrakg}(-2\rho,2\rho - \alpha)[1]$.

Consider the inclusion $\iota : X_{\alpha} := \calB \times_{\calP_{\alpha}} \calB \hookrightarrow \calB \times \calB$. Applying the constructions of \S\ref{ss:base-change-lkd}, we obtain the diagram
\[
\xymatrix@C=2cm{
\calD^{\mathrm{c}}_{G \times \Gm} \bigl( (\Delta \frakg^* \times X_{\alpha}) \, \rcap_{(\frakg^*)^2 \times X_{\alpha}} \, (\wcalN \times_{\calP_{\alpha}} \wcalN) \bigr) \ar[r]^-{\kappa_{\alpha}}_-{\sim} \ar[d]^-{R\hat{\iota}_*} & \calD^{\mathrm{c}}_{G \times \Gm} \bigl( (\overline{\Delta} \frakg^* \times X_{\alpha}) \, \rcap_{(\frakg^*)^2 \times X_{\alpha}} \, (\wfrakg \times_{\calP_{\alpha}} \wfrakg) \bigr)^\op \ar[d]^-{R\tilde{\iota}_*} \\
\calD^{\mathrm{c}}_{G \times \Gm} \bigl( (\Delta \frakg^* \times \calB^2) \, \rcap_{(\frakg^* \times \calB)^2} \, (\wcalN \times \wcalN) \bigr) \ar[r]^-{\kappa}_-{\sim} & \calD^{\mathrm{c}}_{G \times \Gm} \bigl( (\overline{\Delta} \frakg^* \times \calB^2) \, \rcap_{(\frakg^* \times \calB)^2} \, (\wfrakg \times \wfrakg ) \bigr)^\op.
}
\]
Here $\kappa_\alpha$ is associated with the dualizing complex $\iota^!(\omega_\calB \boxtimes \calO_\calB[N]) \cong \calO_{X_\alpha}(-2\rho,2\rho-\alpha)[1]$.
By Proposition~\ref{prop:basechange} there is an isomorphism of functors
\[
\kappa \circ R\hat{\iota}_* \ \cong \ R\tilde{\iota}_* \circ \kappa_{\alpha}.
\]
In particular we obtain an isomorphism
\begin{equation}
\label{eqn:image-Talpha}
\frakK(\calO_{Y_{\alpha}}) \ \cong \ \Xi \circ R\tilde{\iota}_* \circ \kappa_{\alpha}(\calO_{Y_{\alpha}}).
\end{equation}
Here on the right hand side $\calO_{Y_{\alpha}}$ is considered as an object of $\calD^{\mathrm{c}}_{G \times \Gm}\bigl( (\Delta \frakg^* \times X_{\alpha}) \, \rcap_{(\frakg^*)^2 \times X_{\alpha}} \, (\wcalN \times_{\calP_{\alpha}} \wcalN) \bigr)$, with its natural structure of dg-module, and $\Xi$ is defined as in \S\ref{ss:defconvolution}.

Now $Y_{\alpha}$ is a (diagonal) subbundle of $\wcalN \times_{\calP_{\alpha}} \wcalN$. Taking the derived intersection with $\Delta \frakg^* \times X_{\alpha}$, we can apply the results of \S\ref{ss:subbundles}. Denoting by 
\begin{align*}
f : (\Delta \frakg^* \times X_{\alpha}) \, \rcap_{(\frakg^*)^2 \times X_{\alpha}} \, Y_{\alpha} \ & \to \ (\Delta \frakg^* \times X_{\alpha}) \, \rcap_{(\frakg^*)^2 \times X_{\alpha}} \, (\wcalN \times_{\calP_{\alpha}} \wcalN), \\
g : (\overline{\Delta} \frakg^* \times X_{\alpha}) \, \rcap_{(\frakg^*)^2 \times X_{\alpha}} \, (\wfrakg \times_{\calP_{\alpha}} \wfrakg) \ & \to \ (\overline{\Delta} \frakg^* \times X_{\alpha}) \, \rcap_{(\frakg^*)^2 \times X_{\alpha}} \, Y_{\alpha}^{\bot}
\end{align*}
the morphisms of dg-schemes induced by inclusions, we obtain a diagram
\[
\xymatrix@C=2cm{
\calD^{\mathrm{c}}_{G \times \Gm} \bigl( (\Delta \frakg^* \times X_{\alpha}) \, \rcap_{(\frakg^*)^2 \times X_{\alpha}} \, Y_{\alpha} \bigr) \ar[r]^-{\kappa_Y}_-{\sim} \ar[d]^-{Rf_*} & \calD^{\mathrm{c}}_{G \times \Gm} \bigl( (\overline{\Delta} \frakg^* \times X_{\alpha}) \, \rcap_{(\frakg^*)^2 \times X_{\alpha}} \, Y_{\alpha}^{\bot} \bigr)^\op \ar[d]^-{Lg^*} \\
\calD^{\mathrm{c}}_{G \times \Gm} \bigl( (\Delta \frakg^* \times X_{\alpha}) \, \rcap_{(\frakg^*)^2 \times X_{\alpha}} \, (\wcalN \times_{\calP_{\alpha}} \wcalN) \bigr) \ar[r]^-{\kappa_{\alpha}}_-{\sim} & \calD^{\mathrm{c}}_{G \times \Gm}\bigl( (\overline{\Delta} \frakg^* \times X_{\alpha}) \, \rcap_{(\frakg^*)^2 \times X_{\alpha}} \, (\wfrakg \times_{\calP_{\alpha}} \wfrakg) \bigr)^\op }
\]
where $\kappa_Y$ is again associated with the dualizing complex $\calO_{X_\alpha}(-2\rho,2\rho-\alpha)[1]$.
(Here, in the top right corner, $Y_{\alpha}^{\bot}$ is the orthogonal of $Y_{\alpha}$ as a subbundle of $(\frakg^*)^2 \times X_{\alpha}$.) Let $\calY_{\alpha}$ denote the sheaf of sections of $Y_{\alpha}$. The structure sheaf of $(\Delta \frakg^* \times X_{\alpha}) \, \rcap_{(\frakg^*)^2 \times X_{\alpha}} \, Y_{\alpha}$ is $(\bigwedge \frakg) \otimes_{\C} \mathrm{S}_{\calO_{X_{\alpha}}}(\calY_{\alpha}^{\vee})$, with trivial differential (because $Y_{\alpha} \subset \Delta \frakg^* \times X_{\alpha}$). In particular, $\mathrm{S}_{\calO_{X_{\alpha}}}(\calY_{\alpha}^{\vee})$ is naturally an object of the top left category, and $Rf_* (\mathrm{S}_{\calO_{X_{\alpha}}}(\calY_{\alpha}^{\vee}))$ is the object $\calO_{Y_{\alpha}}$ considered above. By Proposition~\ref{prop:inclusion} we have an isomorphism of functors
\[
\kappa_{\alpha} \circ Rf_* \ \cong \ Lg^* \circ \kappa_Y.
\]
In particular we obtain that $\kappa_{\alpha}(\calO_{Y_{\alpha}}) \cong Lg^* \circ \kappa_Y(\mathrm{S}_{\calO_{X_{\alpha}}}(\calY_{\alpha}^{\vee}))$.

Now the structure sheaf of $(\overline{\Delta} \frakg^* \times X_{\alpha})
\, \rcap_{(\frakg^*)^2 \times X_{\alpha}} \, Y_{\alpha}^{\bot}$ is
$(\bigwedge \calY_{\alpha}) \otimes_{\C}
\mathrm{S}(\frakg)$, with trivial differential. And direct computation
shows that
$\kappa_Y(\mathrm{S}_{\calO_{X_{\alpha}}}(\calY_{\alpha}^{\vee}))$
is isomorphic to the dg-module $\calO_{X_{\alpha}}(-2\rho,2\rho-\alpha) \otimes_{\C} \mathrm{S}(\frakg) [1]$. Let $\calF_\alpha$ be the sheaf of sections of  $\wcalN \times_{\calP_{\alpha}} \wcalN$ (i.e.~the restriction of $\calT_{\calB}^\vee \boxplus \calT_{\calB}^\vee$ to $X_\alpha$). Then $\calY_\alpha \subset \calF_\alpha$, and by the same arguments as at the end of the proof of Proposition \ref{prop:diagonal}, $\kappa_{\alpha}(\calO_{Y_{\alpha}})$ can be described as the tensor product of $\calO_{X_{\alpha}}(-2\rho,2\rho-\alpha)[1]$ with the dg-module
\begin{equation}
\label{eqn:Koszul}
\Sym_{\calO_{X_\alpha}} \bigl( \calF_\alpha / \calY_\alpha \to \frakg \otimes_{\C} \calO_{X_\alpha} \bigr).
\end{equation}
(Here, as usual the complex is concentrated in degrees $-1$ and $0$, we identify $\frakg$ with the quotient $(\frakg \oplus \frakg) / \Delta \frakg$, and the differential is the opposite of the natural map.) By the arguments of the proof of \cite[Lemma 4.1.1]{MR}, the latter dg-module is the direct image to $X_\alpha$ of the dg-algebra of functions on the derived intersection $(\overline{\Delta} \frakg^* \times
X_{\alpha}) \rcap_{Y_\alpha^\bot} (\wfrakg \times_{\calP_{\alpha}} \wfrakg)$. Now we observe that $(\overline{\Delta} \frakg^* \times
X_{\alpha}) \cap (\wfrakg \times_{\calP_{\alpha}} \wfrakg)$ is the image of
$\wfrakg \times_{\wfrakg_{\alpha}} \wfrakg$ under the automorphism induced by multiplication by $-1$ on the second copy of $\frakg^*$, and
\[
\dim(\overline{\Delta} \frakg^* \times X_{\alpha}) + \dim(\wfrakg \times_{\calP_{\alpha}} \wfrakg) - \dim(Y_{\alpha}^{\bot}) \ = \ \dim(\wfrakg \times_{\wfrakg_{\alpha}} \wfrakg) \quad (=\dim(\frakg)).
\]
Using \cite[(18.D) Theorem 43 and (16.B) Theorem 31]{matsu}, this implies that the dg-module \eqref{eqn:Koszul} is concentrated in degree $0$, and isomorphic to the sheaf of functions on $(\overline{\Delta} \frakg^* \times
X_{\alpha}) \cap (\wfrakg \times_{\calP_{\alpha}} \wfrakg)$. Using \eqref{eqn:image-Talpha}, this implies the isomorphism of the proposition.
\end{proof} 

\subsection{Proof of Theorem~\ref{thm:mainthm}} \label{ss:endproof}

By construction we have $\frakK(\calM \langle m \rangle) \cong
\frakK(\calM) [m] \langle -m \rangle$, hence \[ \frakK_{\mathrm{IM}}(\calM \langle
m \rangle) \ \cong \ \frakK_{\mathrm{IM}}(\calM) [m] \langle m \rangle. \] In
particular, for $a \in \calH_{\aff}$ and $f(v) \in \mathbb{Z}[v,v^{-1}]$ we have $[\frakK_{\mathrm{IM}}](f(v) \cdot a) =
f(-v) \cdot [\frakK_{\mathrm{IM}}](a)$.

By Proposition~\ref{prop:convolution}, the equivalence $\frakK_{\mathrm{IM}}$ is compatible with convolution, hence also the induced isomorphism $[\frakK_{\mathrm{IM}}]$. Also, by Proposition~\ref{prop:diagonal} it sends the unit to the unit. It follows that to prove Theorem~\ref{thm:mainthm} we only have to check that $[\frakK_{\mathrm{IM}}]$ and $\iota \circ \mathrm{IM}$ coincide on the generators $t_{\alpha}$ and $\theta_x$.

First, Corollary~\ref{cor:imagelinebundle} implies that
$[\frakK_{\mathrm{IM}}](\theta_x)=\theta_{-x}$. Similarly, Proposition~\ref{prop:imageYalpha} implies that we have
$[\frakK_{\mathrm{IM}}]([\calO_{Y_{\alpha}}(\rho-\alpha, -\rho)]) =
-[\calO_{\wfrakg \times_{\wfrakg_{\alpha}} \wfrakg}(-\rho,\rho)]$. Hence
\[
[\frakK_{\mathrm{IM}}](T_{\alpha}) \ = \ -v^{-1} [\calO_{\wfrakg \times_{\wfrakg_{\alpha}} \wfrakg}(-\rho,\rho)] + v^{-1} \ = \ T_{\alpha} - v + v^{-1}.
\]
Hence $[\frakK_{\mathrm{IM}}](t_{\alpha})=-t_{\alpha} + v^2 - 1= -q(t_{\alpha})^{-1}$. This finishes the proof of the theorem.

\end{document}